\documentclass[11pt]{amsart}

\usepackage{geometry}
\usepackage{amsmath,amsfonts, amsthm,amssymb,oldgerm,amssymb,amsthm}
\usepackage{array}
\usepackage{epsfig}
\usepackage{relsize}
\usepackage[usenames, dvipsnames]{color}

\usepackage[pagewise]{lineno} 

\usepackage{marginnote}

\numberwithin{equation}{section}

\usepackage{hyperref}
\usepackage{wasysym}
\newcommand{\sign}{\text{sign}}
\newcommand{\C}{\mathbb C}
\newcommand{\Z}{\mathbb Z}
\newcommand{\F}{\mathcal F}
\newcommand{\R}{\mathbb R}
\newcommand{\N}{\mathbb N}
\newcommand{\X}{\mathcal X}
\newcommand{\D}{\mathcal D}
\newcommand{\be}{\begin{eqnarray}}
\newcommand{\ee}{\end{eqnarray}}
\newcommand{\ben}{\begin{eqnarray*}}
\newcommand{\bc}{\begin{cases}}
\newcommand{\ec}{\end{cases}}
\newcommand{\een}{\end{eqnarray*}}
\newcommand{\ba}{\begin{array}}
\newcommand{\ea}{\end{array}}
\newcommand{\beq}{\begin{equation}}
\newcommand{\eeq}{\end{equation}}
\newcommand{\beqn}{\begin{equation*}}
\newcommand{\eeqn}{\end{equation*}}
\newcommand{\2}{L^2(0,L)}
\newcommand{\izl}{\int_0^L}
\newcommand{\noi}{\noindent}
\newcommand{\dis}{\displaystyle}
\newcommand{\tr}{^\text{tr}}

\newtheorem{theorem}{Theorem}[section]
\newtheorem{open}{Open Problem}[section]

\newtheorem{definition}{Definition}[section]

\newtheorem{proposition}{Proposition}[section]
\newtheorem{claim}{Claim}[section]
\newtheorem{remark}{Remark}[section]
\newtheorem{lemma}{Lemma}[section]

\title{Boundary Effects on the Controllability of Coupled KdV Systems  }

\author{ F. A. Gallego}
\address{\emph{Departamento de Matem\'atica, Universidad Nacional de Colombia (UNAL), Cra 27 No. 64-60, 170003, Manizales, Colombia}}
\email{fagallegor@unal.edu.co}
\author{ A. F. Pazoto}
\address{\emph{Institute of Mathematics, Federal University of Rio de Janeiro, UFRJ, Cidade Universitaria, P.O. Box 68530, CEP 21945-970, Rio de Janeiro, RJ, Brazil.}}
\email{ademir@im.ufrj.br}
\author{I. Rivas}
\address{\emph{Departamento de Matem\'aticas, Universidad del Valle, Calle 13 No. 100 - 00, Ciudadela Universitaria Mel\'endez, Cali, Colombia.}}
\email{ivonne.rivas@correounivalle.edu.co }

\subjclass[2010]{Primary: 35Q53, 93B05 Secondary: 37K10, 30D20}
\keywords{ Gear--Grimshaw system, exact boundary controllability, Hamiltonian and Lagragian systems; Entire functions}

\begin{document}

\begin{abstract}
\noi
We study the exact boundary controllability of a nonlinear coupled system of two Korteweg-de Vries equations on a bounded interval. The model describes the interactions of two weakly nonlinear gravity waves in a stratified fluid. Due to the nature of the system, six boundary conditions are required. However, to study the controllability property, we consider a different combination of the control inputs, with a maximum of four. Firstly, the results are obtained for the linearized system through a classical duality approach and some hidden regularity properties of the boundary terms. This approach reduces the controllability problem to the study of a spectral problem, which is solved by using the Paley–Wiener method introduced by Rosier \cite{rosier}. Then, the issue is to establish when a certain quotient of entire functions still turns out to be an entire function. It can be viewed as a problem of factoring an entire function that, depending on the control configuration, leads to the study of a transcendental equation. Finally, by using the contraction mapping theorem, we derive the local controllability for the full system.
\end{abstract}
\maketitle

\section{Introduction}

In this paper, we are concerned with the study of a nonlinear coupled system derived by Gear and Grimshaw
arising in nonlinear dispersive media \cite{Ge-Gr}. The model describes the strong interactions of two long internal gravity waves in a stratified fluid, where the two waves characterize two different models of the linearized equations of motion. The general structure of the model is  a pair of Korteweg-de Vries (KdV) equations, with linear and  nonlinear coupling terms:
\begin{equation}\label{kdv}
\begin{cases}
u_t + uu_x+u_{xxx} + a v_{xxx} + a_1vv_x+a_2 (uv)_x =0,\\
c v_t +rv_x +vv_x+abu_{xxx} +v_{xxx}+a_2 b uu_x+a_1 b(uv)_x  =0,\\
u(0,x)= u_0(x), \quad v(0,x)=  v_0(x).
\end{cases}
\end{equation}
The parameters $a, a_1, a_2, b, c$ and $r$ are real constants, the unknowns $u$ and $v$
are real-valued functions of the variables $x$ and $t$, and subscripts indicate partial differentiation.
The applicability of the system \eqref{kdv} in a particular
context depends on many factors. Among the more universal of these is that the waves be unidirectional and
essentially two-dimensional. In this case, the study of an initial boundary value problem
arises naturally to model waves generated by a wavemaker mounted at one end of a uniform stretch of the medium.

Our main goal is to study an initial boundary value problem associated with \eqref{kdv} when $x\in [0,L]$ and $t\in \mathbb{R}^{+}$ with the
following boundary conditions:
\begin{equation}\label{inputs}\begin{cases}
u(t,0)=h_0(t),\,\,u_x(t,L)=h_1(t),\,\,u_{xx}(t,L)=h_2(t),\\
v(t,0)=g_0(t),\,\,v_x(t,L)=g_1(t),\,\,v_{xx}(t,L)=g_2(t).
\end{cases}
\end{equation}

The boundary functions $h_i$ and $g_i$, for $i = 0, 1, 2$, are considered as control inputs acting on the boundary conditions.
Their choice was motivated by the previous studies on the boundary controllability properties for the KdV equation developed in \cite{ERZ}.
In fact, since the late 1990s \cite{colin-ghidaglia}, such boundary conditions have been mainly studied
for the well-posedness of the KdV equation in the classical Sobolev spaces $H^s(0, L)$ (see, for instance, \cite{ERZ,colin-ghidaglia,KZ,URZ}).
The results obtained in these works play an important role in our analysis. 

The question to be answered here is the following:\\

{\it Can one drive the solutions of the problem \eqref{kdv}-\eqref{inputs} to have certain desired properties
by choosing appropriate boundary functions? } \\

More precisely, we have the following definition:

\begin{definition} \label{def1}
Let $T>0$. The system \eqref{kdv}-\eqref{inputs} is {\bf exactly controllable in time $T$} if, for any initial state
 $(u_0,v_0) \in (L^2(0,L))^2$ and for any final state  $(u_T,v_T) \in (L^2(0,L))^2$, there exist control functions 
$h_0,g_0\in H^{\frac{1}{3}}(0,T)$,  $h_1,g_1\in L^2(0,T)$ and $h_2,g_2\in H^{-\frac{1}{3}}(0,T)$, such that the solution $(u,v)$ of \eqref{kdv}-\eqref{inputs} satisfies    $$u(T,\cdot)=u_T  \quad \text{and}  \quad v(T,.)=v_T.$$
\end{definition}

The study of the controllability properties for system \eqref{kdv} was first addressed in \cite{MOP07} with the following set of boundary conditions:
\begin{equation}\label{inputsAll-1}
\begin{cases}
u(t,0)=0,\,\,u(t,L)=h_1(t),\,\,u_{x}(t,L)=h_2(t),\\
v(t,0)=0,\,\,v(t,L)=g_1(t),\,\,v_{x}(t,L)=g_2(t).
\end{cases}
\end{equation}

To prove the results, the authors combine the controllability of the linearized system and a fixed point argument.  For the linear case, the controllability was obtained through the classical duality arguments \cite{dolecki}, \cite{lions}, \cite{rosier}, where
the controllability problem is equivalent to establishing an observability inequality for the solutions of the corresponding adjoint system. In this case, a nonstandard unique continuation principle for the eigenfunctions of the differential operator associated with the model is required. As in \cite{rosier}, this is done by combining Fourier analysis and Paley-Winer theorem. Later on, in \cite{cerpa-pazoto}, the authors proved that, in some situations, it is possible to get the controllability of the linearized system by using only two controls on the Neumann boundary conditions, $h_2$ and $g_2$. Indeed, to prove the observability inequality, they used a direct approach based on the multiplier technique, which gives the observability inequality for small values of the length $L$ and large time of control $T$. In both works, the controllability of the nonlinear system runs exactly in the same way, using a fixed point argument.\\

More recently, combining the approach developed in \cite{rosier} with some hidden regularity properties of the boundary terms, the exact controllability property for \eqref{kdv} was studied in \cite{Capistrano} by considering the following set of boundary conditions:
\begin{equation}\label{inputsAll}
\begin{cases}
u_{xx}(t,0)=h_0,\,\,u_{x}(t,L)=h_1(t),\,\,u_{xx}(t,L)=h_2(t),\\
v_{xx}(t,0)=g_0,\,\,v_{x}(t,L)=g_1(t),\,\,v_{xx}(t,L)=g_2(t).
\end{cases}
\end{equation}
Depending on the configuration of the controls, the controllability property of \eqref{kdv}-\eqref{inputsAll} is obtained with some restriction over the size of the interval. More precisely, the controllability holds whenever $L$ does not belong to a countable set of critical lengths. Later on, by considering the boundary conditions given by \eqref{inputsAll-1}, the control problem for \eqref{kdv} was addressed in \cite{Capistrano2}. Proceeding as in \cite{Capistrano}, the authors proved that, with another configuration of four controls, the controllability of the system may depend on the length of the spatial domain.

In the works mentioned above \cite{Capistrano,Capistrano2,MOP07}, the controllability problem is addressed by following closely the ideas
introduced in \cite{rosier} while studying the boundary controllability of the KdV equation posed on a bounded interval $(0,L)$.
Being different from other systems, the length $L$ of the spatial domain may play a crucial role in
determining the controllability of the system, especially when some configurations of the control inputs are allowed to be used. This phenomenon, the so-called {\it critical length phenomenon}, was
observed for the first time in \cite{rosier}. Roughly speaking, it was proved the existence of a
finite dimensional subspace $M$ of $L^2(0,L)$, which is not reachable by the linearized KdV equation, when
starting from the origin, if $L$ belongs to a countable set of lengths.

In this paper, we proceed as in \cite{MOP07}, \cite{rosier} to extend the analysis of the controllability properties for system
\eqref{kdv}-\eqref{inputs}. By considering different combinations of the boundary controls $g_i$ and $h_i$, with $i = 0, 1, 2$, we prove that the controllability properties hold for any length $L$ of the spatial domain. This is done by combining the controllability of the linearized system and a fixed point argument in an appropriate function space. As in the previous works, the controllability of the linearized system is equivalent to the proof of an observability inequality for the solutions of the corresponding adjoint system. To prove such inequality, we derived linear estimates, including hidden regularity (sharp trace regularity) results for the solutions of the adjoint system. In some cases, however, this is not enough due to some technical difficulties we have to deal with. Indeed, proceeding as in \cite{rosier}, the above arguments reduce the problem to show a unique continuation result for the state operator. To prove the result, we extend the solution under consideration by zero in $\mathbb{R} \setminus [0, L]$ and take the Fourier transform. Then, the issue is to establish when a certain quotient of entire functions still turns out to be an entire function, which leads to the study of the transcendental equation
\begin{equation}\label{zez}
ze^z = \alpha.
\end{equation}
Equation  \eqref{zez} has been studied by a large number of authors (see, for instance, \cite{corless}, \cite{siewert}), and its solutions are given by the Lambert function $W(z)$ which satisfies
$$W(z)e^{W(z)} = \alpha.$$
Moreover, if $\alpha\neq 0$, equation \eqref{zez} has infinitely many solutions, which are given by the zeros of the holomorphic complex functions
$$W_k(z) = \log(\alpha) – z –\log(z) + 2k\pi i, k\in \mathbb{Z},$$ 
where $\log(z)$ denotes the principal branch of the log-function. This approach allows us to conclude that, in some cases, the controllability holds for any $L>0$.

Our analysis over some of the coefficients in the system is based on the assumption made in \cite{bona}, \cite{saut}, where the parameters $b$ and $c$ are automatically positive and $r$ is a non-dimensional parameter that could be assumed very small. However, as it will become clear throughout the paper, to provide the tools to handle this problem, we assume that
\begin{equation}\label{constantscondition}
 b, c > 0 \quad  \text{and} \quad  0<1-a^2 b.
\end{equation}

Our main results can be summarized as follows.

\begin{theorem}\label{maintheorem}
 Let \(T>0\), \(L>0 \) and $r\neq 0$. Then, there exists $\delta=\delta(t,L)>0$, such that, for any $(u_0,v_0)$, $ (u_T,v_T) \in\X:=(L^2(0,L))^2$ verifying
$$\|(u_0,v_0)\|_{\X} + \|(u_T,v_T)\|_{\X}\le \delta,$$
there exist controls $(h_0,h_1,h_2)$ and $(g_0,g_1,g_2)$ in $\mathcal{H}:=H^{\frac13}(0,T)\times L^2(0,T)  \times H^{-\frac13}(0,T)$, with the following configurations for four controls:
\begin{enumerate}
\item[(i)]  $ h_0\neq 0, h_1\neq 0, h_2\neq 0 \quad \text{and} \quad  g_0= 0, g_1\neq 0, g_2=0$,
\item[(ii)]  $ h_0= 0, h_1\neq 0, h_2= 0 \quad \text{and} \quad  g_0\neq 0, g_1\neq 0, g_2\neq0$,
\item[(iii)]  $ h_0\neq 0, h_1\neq 0, h_2= 0 \quad \text{and} \quad  g_0\neq 0, g_1\neq 0, g_2=0$,
\item[(iv)]  $ h_0= 0, h_1\neq 0, h_2\neq 0 \quad \text{and} \quad  g_0= 0, g_1\neq 0, g_2\neq0$,
\end{enumerate}
 such that the system \eqref{kdv}-\eqref{inputs} under condition \eqref{constantscondition}, admits a unique solution
 $$(u,v) \in C\left([0,T];(L^2(0,L))^2\right)\cap L^2\left(0,T,(H^1(0,L))^2\right)$$
 satisfying
\vspace{-0.2cm}
 \begin{gather*}
u(T,x)=u_T(x)\quad\mbox{ and }\quad v(T,x)=v_T(x).
\end{gather*}
 Moreover, if the parameters $a,b,c$ satisfy \eqref{constantscondition} and
$$0 <	C_1(1-a^2b)<c,$$
 where $C_1=C_1(T,L)>0$ is the hidden regularity constant given in \eqref{putaquepario2}, the same result is obtained with the following configuration for three controls:
 
\begin{enumerate}
\item[(v)]  $ h_0\neq 0, h_1\neq 0, h_2= 0 \quad \text{and} \quad  g_0\neq 0, g_1= 0, g_2=0$,
\item[(vi)]  $ h_0\neq 0, h_1= 0, h_2= 0 \quad \text{and} \quad  g_0\neq 0, g_1\neq 0, g_2=0$.
\end{enumerate}
 \end{theorem}

The additional condition of the parameters $a,b$ and $c$ to prove the exact controllability for three controls allows the absorption of some boundary terms, leading to linear a priori estimates that play an important role in the proofs.\\

Following the ideas involved in the proof of Theorem \ref{maintheorem}, it is possible to obtain similar results with
different configurations of three and four controls. We also remark that the corresponding linear system can be written
in an equivalent diagonal form and, when the parameter $r=0$,  we obtain two independent KdV equations. In this case, the model
can be analyzed by using the same theory, which leads to the desired results for the corresponding linear system. We address both issues in the next sessions.\\

In addition, Theorem \ref{maintheorem} establishes as a fact that system \eqref{kdv}-\eqref{inputs} inherits the interesting controllability properties initially observed in \cite{ERZ} for the KdV equation. As far as we know, the problem we address here has not been addressed in the literature yet. Since the system is very sensitive to the choice of boundary conditions, the existing results do not give an immediate answer to it, including \cite{Capistrano} and \cite{Capistrano2}.\\

The remainder of this paper is organized as follows. Section \ref{well-posedness} covers the study of well-posedness for the linear system associated with \eqref{kdv}, its adjoint, and the full system. Additionally, we present various linear estimates, among them some hidden regularities results, which are crucial tools to prove our results.  Section \ref{observabilildade1} is devoted to the observability conditions related to the adjoint system when different control configurations are considered. In Section \ref{exactcontrollability}, we give a positive answer to the linear controllability problem and extend the result to the full system via the contraction mapping theorem. Finally, Section \ref{furthercommentsopenproblems} is dedicated to further comments and open problems.

\section{ Well-posedness}\label{well-posedness}

We first introduce the notation that will be used in the whole paper. From now on, we consider the product of Sobolev spaces
$$\mathbf{\mathcal{H}} :=H^{\frac13}(0,T)\times L^2(0,T)  \times H^{-\frac13}(0,T),$$
endowed with their natural inner products for the boundary data. We also introduce the space
$$\mathcal{X}:=(L^2(0,L))^2,$$
for the initial data space, endowed with the inner product
\begin{equation}
\label{eq:prod-1} \left\langle (u,v) , (\varphi
,\psi)\right\rangle := \frac{b}{c}\int_0^L u\varphi dx + \int_0^L v\psi dx.
\end{equation}
\subsection{Linear Systems} In this section, we analyze the well-posedness of both the linear system
associated to \eqref{kdv}-\eqref{inputs} and the corresponding adjoint system ($h_1=g_i\equiv 0$, $i=0,1,2$).
We also obtain some hidden (sharp trace) regularities results together with linear estimates for the solutions.\\

In the sequel, $C$ denotes a generic positive constant; $C_0, C_1, \ldots$, etc. are other positive (specific) constants. The results obtained for the KdV equation will be used in our proofs.

\subsubsection{\bf The IBVP for the KdV equation}
In \cite{ERZ}, the authors study the well-posedness of the initial boundary value problem for the KdV equation posed on a bounded interval $[0,L]$ for $s>-1$, when the function $f$ is in a lower regularity space is treated as the nonlinear term:
\begin{equation}\label{kdvlin1}
\begin{cases}
w_t + w_{xxx} + w_{x}  =f, \\
w(t,0)= k_1(t),\,w_{x}(t,L)=k_2(t),\, w_{xx}(t,L)=k_3(t),\\
w(0,x)=w_{0}.
\end{cases}
\end{equation}
and as a direct consequence, Propositions 2.5, 2.6, and 2.8 in \cite{ERZ}, can be adapted directly  for the problem considering  $s=0$,
\begin{equation}\label{kdvlinaux}
\begin{cases}
w_t + w_{xxx} =f, \\
w(t,0)= k_1(t),\,w_{x}(t,L)=k_2(t),\, w_{xx}(t,L)=k_3(t),\\
w(0,x)=w_{0}.
\end{cases}
\end{equation}
More precisely,
\begin{proposition}\label{linearNonr-kdv} Let $T>0$ be given. Then,  for any ${w}_0\in L^2(0,L)$, $f \in L^1(0,T;L^2(0,L))$
and boundary data $(k_1,k_2,k_3) \in \mathbf{\mathcal{H}} $, system \eqref{kdvlinaux} admits a unique solution
$$w \in Z_T := C([0,T]; L^{2}(0,L))\cap L^{2}(0,T;H^{1}(0,L)),$$
which, in addition, has the hidden (or sharp trace) regularities
$$\partial_x^j w \in L^{\infty}_x(0,L; H^{\frac{1-j}{3}}(0,T)), \quad j=0, 1, 2.$$
Moreover, there exists $C>0$, such that
$$||w||_{Z_T} + \sum_{j=0}^2 ||\partial_x^j w||_{L^{\infty}_x(0,L; H^{\frac{1-j}{3}}(0,T))} \leq
C\{||w_0||_{L^{2}(0,L)} + ||(k_1,k_2,k_3)||_{\mathbf{\mathcal{H}}} + ||f||_{L^1(0,T;L^{2}(0,L)) }\}.$$
\end{proposition}

\subsubsection{\bf The system of two couple KdV equations}
We consider the direct system
\begin{equation}\label{kdvlin}
\begin{cases}
u_t + u_{xxx} + av_{xxx}  =0, \\
c v_t +rv_x +abu_{xxx} + v_{xxx} =0,
\end{cases}
\end{equation}
with the following initial and boundary data
\begin{gather}\label{kdvlinconditions-1}
u(0,x)= u_0(x) \quad \text{and} \quad  v(0,x)  = v_ 0(x),
\end{gather}
\begin{gather}\label{kdvlinconditions}
\begin{cases}
u(t,0) =h_{0}(t),\, u_x(t,L) =h_{1}(t), \, u_{xx}(t,L)  = h_{2}(t) ,\\
v(t,0) =g_{0}(t),\, v_{x}(t,L) =g_{1}(t), \, v_{xx}(t,L)  =g_{2}(t).\\
\end{cases}
\end{gather}
When $h_i=g_i=0$, $\forall i=0,1,2$, the operator associated with the space variable is dissipative with the norm defined in \eqref{eq:prod-1}.  Nevertheless,  it is not clear that the corresponding adjoint operator is also dissipative.  Hence, the classical semi-group theory does not work directly over the system \eqref{kdvlin}-\eqref{kdvlinconditions}. To overcome this difficulty, we apply a fixed-point argument. Therefore, we first introduce the following non-homogeneous coupled  system associated with \eqref{kdvlin}
\begin{equation}\label{kdvlinforc}
\begin{cases} 
u_t + u_{xxx} + av_{xxx}  =p(t,x), \\
c v_t + abu_{xxx} + v_{xxx} =q(t,x), 
\end{cases}
\end{equation}
with initial and boundary data given by  \eqref{kdvlinconditions-1} and \eqref{kdvlinconditions}, respectively.
\begin{remark}\label{decouple}

By using the change of variable, provided in  \cite[Remark 2.1]{MOP07},
\begin{equation*}\label{a}
  \begin{cases}
  u=2a\tilde{u} + 2a\tilde{v}\\
  v=((\frac{1}{c}-1)+\lambda)\tilde{u} + ((\frac{1}{c}-1)-\lambda)\tilde{v},
  \end{cases}
\end{equation*}

with $\lambda=\sqrt{(\frac{1}{c}-1)^2 + \frac{4a^2 b}{c}}>0$,  allows us to transform the original  system \eqref{kdvlinforc}  into two independent KdV equations given by
\begin{gather} \label{kdvdecoupleforcing}
\begin{cases} 
\bar{u}_t + \alpha^+\bar{u}_{xxx}=\bar{p}, \\
 \vspace{2mm} \bar{v}_t +\alpha^-\bar{v}_{xxx} =\bar{q},
 \end{cases}
\end{gather}

The initial conditions will be called
$$\bar{u}(0,x)= \bar{u}_0(x), \quad \bar{v}(0,x)  = \bar{v}_0(x),$$
and  boundary conditions
\begin{equation}\label{kdvlin2condicions}
\begin{cases}
\bar{u}(t,0) = \bar{h}_0, \quad  \bar{u}_x(t,L) = \bar{h}_1,  \quad  \bar{u}_{xx}(t,L)  = \bar{h}_2, \\
\bar{v}(t,0) = \bar{g}_0, \quad \bar{v}_x(t,L)  =  \bar{g}_1, \quad \bar{v}_{xx}(t,L)  = \bar{g}_2,
\end{cases}
\end{equation}
\\   
the parameters $\alpha^+,\alpha^- \in \R$ are given by
\begin{equation}
\alpha_+=\frac{\frac{1}{c}+1+\lambda}{2}\qquad\mbox{ and }\qquad\alpha_-=\frac{\frac{1}{c}+1-\lambda}{2},
\end{equation}
notice that 
$\alpha_+>0$ since, $\lambda$ and $c$ are non negative numbers, and since $1-a^2b>0$ together with the definition of $\lambda$, by direct computations we conclude that $\alpha_->0.$   Hence, the wellposedness of system \eqref{kdvdecoupleforcing} is a direct consequence of Proposition \ref{linearNonr-kdv} and so, the 
 system \eqref{kdvlinforc} through the change of variables is also well-posed.\end{remark}

Therefore, the well-posedness of the system \eqref{kdvlinforc} with conditions \eqref{kdvlinconditions-1}--\eqref{kdvlinconditions}
can be stated as follows:
\begin{proposition}\label{linearNonr} Let $T>0$ be given. Then,  for any $({u}_0,{v}_0)\in \X$, ${p},{q} \in L^1(0,T;L^2(0,L))$ and boundary data $\vec{h}, \vec{g} \in \mathbf{\mathcal{H}} $, system \eqref{kdvlinforc}, \eqref{kdvlinconditions-1}--\eqref{kdvlinconditions}  admits a unique solution
$$(u,v) \in X_T := C([0,T]; \X)\cap L^{2}(0,T;(H^{1}(0,L))^2),$$
which, in addition, has the hidden (or sharp trace) regularities
$$\partial_x^j u, \partial_x^j v \in L^{\infty}_x(0,L; H^{\frac{1-j}{3}}(0,T)), \quad j=0, 1, 2.$$
Moreover, there exists $C>0$, such that
\begin{multline}
||(u,v)||_{X_T} + \sum_{j=0}^2 ||(\partial_x^j u,\partial_x^j v)||_{L^{\infty}_x(0,L; (H^{\frac{1-j}{3}}(0,T))^2)} \leq
C\{||(u_0,v_0)||_{\X}  + ||(\vec{h},\vec{g})||_{\mathbf{\mathcal{H}}^2} \\
+ ||(p,q)||_{L^1(0,T; \X)}\}.
\end{multline}
\end{proposition}
\begin{proof}
By Remark \ref{decouple}, system \eqref{kdvlinforc}, \eqref{kdvlinconditions-1}--\eqref{kdvlinconditions} can be decoupled into two independent initial boundary value problems for the KdV equation as follows
 \begin{equation}\label{chanvar}
\left\{
\begin{array}{ll}
\bar{u}_t + \alpha^+\bar{u}_{xxx} =\bar{p},\\
\bar{u}(t,0) = \bar{h}_0(t), \\  \bar{u}_x(t,L) = \bar{h}_1(t),  \\
\bar{u}_{xx}(t,L)  = \bar{h}_2(t),\\
\bar{u}(0,x)= \bar{u}^0(x)
\end{array}
\right. \quad  \text{and}  \quad
\left\{
\begin{array}{ll}
\bar{v}_t +\alpha^-\bar{v}_{xxx}=\bar{q},  \\
\bar{v}(t,0) = \bar{g}_0(t), \\ \bar{v}_x(t,L)  =  \bar{g}_1(t), \\
\bar{v}_{xx}(t,L)  = \bar{g}_2(t), \\
\bar{v}(0,x)  = \bar{v}^0(x).
\end{array}
\right.
\end{equation}
Moreover, 
$\bar{p},\bar{q}\in L^1(0,T;L^2(0,L))$, $\vec{h},\vec{g}\in \mathbf{\mathcal{H}}$ and $\bar{u}^0,\bar{v}^0\in L^2(0,L)$, respectively. Then, the result follows from Propositions \ref{linearNonr-kdv}. We also remark that, if $\bar{p}=\bar{u}^0 = 0$, the solution $\bar{u}$ of \eqref{chanvar} can be written as
\begin{equation}
\bar{u}(t)=S_{bdr} (t)\vec{h},
\end{equation}
where $\vec{h}:=(\bar{h}_0,\bar{h}_1,\bar{h}_2)$. The operator $S_{bdr} (t)$ is the called boundary integral operator, whose explicit representation can be found in \cite{KRZ}, \cite{KZ}. For the solution $\bar{v}$ we have a similar representation.
\end{proof}
Having Proposition \ref{linearNonr} in hand, the local well-posedness of the system \eqref{kdvlin}--\eqref{kdvlinconditions} holds. 

\begin{proposition}\label{linearNonr-1} Let $T>0$ be given. Then,  for any $({u}_0,{v}_0)\in \X$ and boundary data
$\vec{h}, \vec{g} \in \mathbf{\mathcal{H}} $, system \eqref{kdvlin}--\eqref{kdvlinconditions}
admits a unique solution
$$(u,v) \in X_T := C([0,T]; \X)\cap L^{2}(0,T;(H^{1}(0,L))^2),$$
which, in addition, has the hidden (or sharp trace) regularities
$$\partial_x^j u, \partial_x^j v \in L^{\infty}_x(0,L; H^{\frac{1-j}{3}}(0,T)), \quad j=0, 1, 2.$$
Moreover, there exists $C>0$, such that
$$||(u,v)||_{X_T} + \sum_{j=0}^2 ||(\partial_x^j u,\partial_x^j v)||_{L^{\infty}_x(0,L; (H^{\frac{1-j}{3}}(0,T))^2)} \leq
C\{||(u_0,v_0)||_{\X} + ||(\vec{h},\vec{g})||_{\mathbf{\mathcal{H}}^2}\}.$$
\end{proposition}

For the proof of Proposition \ref{linearNonr-1}, we employ a classical fixed-point
argument following the ideas in \cite[Proposition 2.3]{Capistrano}. For the sake of completeness, the proof has been included in the Appendix \ref{moreresults}.

\subsection{The adjoint system}
The second system to be analyzed is the adjoint system associated to \eqref{kdvlin}, that is,
\begin{equation}\label{kdvadjoin}
\begin{cases} 
\varphi_t + \varphi_{xxx} + a\psi_{xxx}  =0, \\
c \psi_t + ab\varphi_{xxx} + r \psi_{x} + \psi_{xxx} =0, 
\end{cases}
\end{equation}
with data
\begin{gather*}
\varphi(T,x)= \varphi_T(x) \quad \text{and} \quad \psi(T,x)  = \psi_T(x),
\end{gather*}
and boundary conditions
\begin{gather}\label{kdvadjcondition}
\begin{cases}
\varphi(t,0) =\varphi_x(t,0) =\psi(t,0) =\psi_{x}(t,0) =0,\\
 \varphi_{xx}(t,L)+a \psi_{xx}(t,L)=0, \\
 a b\varphi_{xx}(t,L)+{r} \psi(t,L)+ \psi_{xx}(t,L)=0.
\end{cases}
\end{gather}
Observe that the  two last conditions of \eqref{kdvadjcondition} imply that
  $$(1-a^{2}b)\psi_{xx}(t,L)+r \psi(t,L)=0.$$
Then, the main result of this subsection reads as follows:
\begin{proposition}\label{welladj} Let $T>0$ be given. Then, for any $(\varphi_T,v_T)\in \X$ system \eqref{kdvadjoin}-\eqref{kdvadjcondition} admits a unique solution   $(\varphi, \psi) \in X_T$. Moreover,
\begin{gather}\label{putaquepario2}
\begin{cases}
\sup_{ 0<x<L} \| \partial_{x}^{j} \varphi(\cdot,x) \|_{H^{\frac{1-j}{3}}(0,T)} \le C_{j}  \| \varphi_T\|_{L^{2}(0,L)},\\
\sup_{ 0<x<L} \| \partial_{x}^{j} \psi(\cdot,x) \|_{H^{\frac{1-j}{3}}(0,T)} \le C_{j}  \| \psi_T\|_{L^{2}(0,L)},
\end{cases}
\end{gather}
for some $C_j>0$, where $j=0,1,2$.  Moreover, we have that
\begin{multline}\label{putaquepario}
\|(\varphi_T, \psi_T)\|_{\mathcal{X}}^2  \leq    \frac{1}{T}\|(\varphi,\psi)\|^2_{L^2(0,T;\mathcal{X})}  +  \frac{1}{c}\int_0^T   \psi^2(t,L)dt  \\
+ \frac{1}{c}  \int_0^T   \left(  b \varphi_x^2(t,L) + 2ab \psi_x(t,L)  \varphi_{x}(t,L) +\psi^2_x(t,L) \right)  dt.
\end{multline}
\end{proposition}
\begin{proof}
If we consider the change of variables $x\mapsto L-x$ and $t\mapsto T-t$, system \eqref{kdvadjoin}-\eqref{kdvadjcondition}  become
 \begin{equation}\label{kdvadjoint}
\begin{cases}
\varphi_t + \varphi_{xxx} + a\psi_{xxx}  =0, \\
c \psi_t +ab\varphi_{xxx} + r \psi_{x} + \psi_{xxx} =0, 
\end{cases}
\end{equation}
with initial conditions
$$\varphi(0,x)= \varphi_0(x) \quad \text{and} \quad \psi(0,x)  = \psi_0(x),$$
and boundary data
\begin{gather}\label{kdvadjconditions}
\begin{cases}
\varphi(t,L) =\varphi_x(t,L) =\psi(t,L) =\psi_{x}(t,L) =0,\\
 \varphi_{xx}(t,0)= \nu(t),\\
 \psi_{xx}(t,0)= \xi(t) ,
\end{cases}
\end{gather}
whit $\nu(t):=- a \psi_{xx}(t,0)$ and  $\xi(t) := - ac\varphi_{xx}(t,0)- r\psi(t,0)$. 

The proof used a fixed point argument as in Proposition \ref{linearNonr-1}, So we omit it.  On the other hand, multiplying the first equation of \eqref{kdvadjoin} by \(-t \varphi,\) the second one by \(-\frac{b}{c} t \psi\) and integrating by parts over \((0, T) \times(0, L),\) we obtain
\begin{multline*}
\frac{Tb	}{2c}\int_0^L \varphi^2(T,x)dx = \frac{b}{2c} \int_0^T\int_0^L \varphi^2dxdt + \frac{b}{2c}\int_0^Tt \varphi_x^2(t,L)dt -\frac{b}{c}\int_0^Tt \varphi_{xx}(t,L)\varphi(t,L)dt
\\
+ \frac{ab}{c}\int_0^Tt \psi\varphi_{xxx}dxdt -\frac{ab}{c}\int_0^Tt \left[\psi(t,L)  \varphi_{xx}(t,L)- \psi_x(t,L)  \varphi_{x}(t,L)+ \psi_{xx}(t,L)  \varphi(t,L)\right] dt
\end{multline*}
and
\begin{multline*}
\frac{T}{2}\int_0^L \psi^2(T,x) dx = \frac{1}{2} \int_0^T\int_0^L \psi^2dxdt - \frac{ab}{c}\int_0^Tt \psi\varphi_{xxx}dxdt  - \frac{r}{2c}\int_0^T t \psi^2(t,L)dt \\
+ \frac{1}{2c}	\int_0^T t \psi^2_x(t,L)dt-\frac{1}{c}\int_0^T t \psi_{xx}(t,L)\psi(t,L)dt.
\end{multline*}
Adding the above identities, it yields
\begin{multline*}
\frac{T}{2} \left( \frac{b	}{c}\int_0^L \varphi^2(T,x)dx +\int_0^L \psi^2(T,x) dx \right)=  \frac{1}{2}\|(\varphi,\psi)\|^2_{L^2(0,T;\mathcal{X})} + \int_0^T t \left\lbrace  \frac{b}{2c} \varphi_x^2(t,L)   \right. \\
\left. -\frac{b}{c}\varphi_{xx}(t,L)\varphi(t,L) - \frac{ab}{c} \left[\psi(t,L)  \varphi_{xx}(t,L)- \psi_x(t,L)  \varphi_{x}(t,L)+ \psi_{xx}(t,L)  \varphi(t,L)\right]	\right. \\
\left. -  \frac{r}{2c}\psi^2(t,L) + \frac{1}{2c}\psi^2_x(t,L) -\frac{1}{c}\psi_{xx}(t,L)\psi(t,L) \right\rbrace dt.
\end{multline*}
Thus,
\begin{multline*}
\frac{T}{2} \left( \frac{b	}{c}\int_0^L \varphi^2(T,x)dx +\int_0^L \psi^2(T,x) dx \right)=  \frac{1}{2}\|(\varphi,\psi)\|^2_{L^2(0,T;\mathcal{X})} \\
+ \int_0^T t \left\lbrace  \frac{1}{2c} \left(  b \varphi_x^2(t,L) + 2ab \psi_x(t,L)  \varphi_{x}(t,L) +\psi^2_x(t,L) \right) - \frac{b}{c}\varphi(t,L)  \left(  \varphi_{xx}(t,L)+ a\psi_{xx}(t,L) \right)  \right. \\
\left.-  \frac{1}{c} \psi(t,L)  \left(  ab\varphi_{xx}(t,L)+ \psi_{xx}(t,L) \right)  -  \frac{r}{2c}\psi^2(t,L)  \right\rbrace dt.
\end{multline*}
Applying the boundary  conditions, we have that
\begin{multline*}
\frac{T}{2} \left( \frac{b	}{c}\int_0^L \varphi^2(T,x)dx +\int_0^L \psi^2(T,x) dx \right)  =   \frac{1}{2}\|(\varphi,\psi)\|^2_{L^2(0,T;\mathcal{X})}
\\ + \frac{1}{2c}  \int_0^T t   \left(  b \varphi_x^2(t,L) + 2ab \psi_x(t,L)  \varphi_{x}(t,L) +\psi^2_x(t,L) \right)  dt+  \frac{1}{2c}\int_0^T t  \psi^2(t,L)dt.
\end{multline*}
\end{proof}
\subsection{Nonlinear System}

Using the standard fixed point argument, we can establish the well-posedness for the full system. For completeness, the proof has been included in the Appendix \ref{moreresults}.

\begin{theorem}\label{nonlinearwell-posedness}
Let $T>0$ be given. Then,  for any $({u}_0,{v}_0)\in \X$ and boundary data
$\vec{h}, \vec{g} \in \mathbf{\mathcal{H}} $, there exists $T^*>0$, depending on  $\|({u}_0,{v}_0)\|_{\X}$, such that the system \eqref{kdv}--\eqref{inputs} admits a unique solution
$$(u,v) \in X_{T^*}:= C([0,T^*]; \X)\cap L^{2}(0,T^*;(H^{1}(0,L))^2),$$
which, in addition, has the hidden (or sharp trace) regularities
$$\partial_x^j u, \partial_x^j v \in L^{\infty}_x(0,L; H^{\frac{1-j}{3}}(0,T^*)), \quad j=0, 1, 2.$$
Moreover,  the corresponding solution map is Lipschitz continuous.
\end{theorem}

\section{Observability Properties}\label{observabilildade1}

We  study  the  exact controllability property  when   the system   \eqref{kdvlin}-\eqref{kdvlinconditions}  starts from zero, i.e.,
when $u_0\equiv v_0\equiv 0$. For another initial state  $(u_0,v_0) \in \X$  the linearity of the system together with the translation   $(u,v)+(u_0,v_0)$ will let us recover the solution.\\

To understand the interaction between the number of controls and the controllability properties,
we consider the more general case $\mathbf{r\neq 0}$  for  the linear system  \eqref{kdvadjoin}-\eqref{kdvadjcondition}.
Some comments about the case  $r=0$ are presented in section \ref{furthercommentsopenproblems} below. \\

A positive answer to the exact control problem is obtained by using the duality approach \cite{dolecki,lions}. In this case, the observability inequality of a dual problem plays a fundamental role. Thus, proceeding as in \cite{rosier}, we show that the observability inequality holds when several configurations of boundary controls
are considered.

\subsection{Case A: Observability inequality with four control configurations}
In this subsection, we consider the following boundary control configuration ($g_0=g_2=0$ and $h_0 \neq 0,   h_1 \neq 0, h_2 \neq 0$, $g_1 \neq 0$):
\begin{equation*}
\left\lbrace
\begin{array}{llll}
u(t,0)=h_0(t), & u_x(t,L)=h_1(t), & u_{xx}(t,L)=h_2(t), & \text{in $(0,T)$,}\\
v(t,0)=0, & v_x(t,L)=g_1(t), & v_{xx}(t,L)=0, & \text{in $(0,T)$}.
\end{array}\right.
\end{equation*}
In this case, we establish the following observability inequality:
\begin{proposition}\label{prop1}
 Let $T>0$ and  $L > 0$. Then, there exists  $C=C(L,T)>0$,  such that, if $(\varphi,\psi) \in C([0,T];L^2(0,L))^2 $ is a solution of  \eqref{kdvadjoin}-\eqref{kdvadjcondition},  the inequality
\begin{multline*}
\| \psi_T\|^2_{L^2(0,L)} + \frac{b}{c} \| \varphi_T\|^2_{L^2(0,L)} \le C \left( \|  \varphi(t,L) +a \psi(t,L) \|^{2}_{H^{\frac13}(0,T)}
\right. \\
 + \|  \varphi_x(t,L) + a\psi_x(t,L) \|^{2}_{L^2(0,T)}   + \|  ab\varphi_x(t,L) + \psi_x(t,L) \|^{2}_{L^2(0,T)}   \\
\left. + \|  \varphi_{xx}(t,0) + a\psi_{xx}(t,0) \|^{2}_{H^{-\frac13}(0,T)}\ \right)
\end{multline*}
holds for any $\varphi_T, \psi_T \in L^2(0,L)$.
\end{proposition}
\begin{proof}
The proof follows the ideas developed in \cite{rosier}.  Arguing by contradiction, we obtain a sequence o functions $\{(\varphi^n_T,\psi^n_T)\}_{n\in \N}$, such that
\begin{multline}\label{normalized_1}
1=\|(\varphi^n_T,\psi^n_T)\|_{\mathcal{X}} > n \left( \| \varphi^n(t,L) +a \psi^n(t,L)\|^{2}_{H^{\frac13}(0,T)} \right. \\
 + \|  \varphi^n_x(t,L) +a \psi^n_x(t,L) \|^{2}_{L^2(0,T)}  + \|  ab\varphi_x^n(t,L) + \psi^n_x(t,L) \|^{2}_{L^2(0,T)}   \\
\left.+\|  \varphi^n_{xx}(t,0) + a\psi^n_{xx}(t,0)  \|^{2}_{H^{-\frac13}(0,T)}\ \right).
\end{multline}
Then,
\begin{equation}\label{new1}
\begin{cases}
 \varphi^n(t,L) +a\psi^n(t,L)  \rightarrow 0, & \text{in $H^{\frac13}(0,T)$ }, \\
\varphi^n_x(t,L) +a \psi^n_x(t,L)\rightarrow 0, & \text{in $L^2(0,T)$,} \\
ab\varphi^n_x(t,L) + \psi^n_x(t,L)\rightarrow 0, & \text{in $L^2(0,T)$,} \\
\varphi^n_{xx}(t,0) + a\psi^n_{xx}(t,0)   \rightarrow 0, & \text{in $H^{-\frac13}(0,T)$,}
\end{cases}
\end{equation}
as $n\rightarrow \infty$.   On the other hand, Proposition \ref{welladj} guarantees that $\{(\varphi^n,\psi^n)\}_{n\in \N}$ is bounded in $L^2(0,T; (H^1(0,L))^2)$ and, using the equations of the system,  we deduce that $\{(\varphi_t^n,\psi_t^n)\}_{n\in \N}$ is bounded in $L^2(0,T; (H^{-2}(0,L))^2)$. From the compact embedding  \(H^{1}(0, L) \hookrightarrow L^{2}(0, L) \hookrightarrow H^{-2}(0, L)\) and by using compactness arguments, it follows that $\{(\varphi^n,\psi^n)\}_{n\in \N}$ is relatively compact in $L^2(0,T; \mathcal{X})$. Hence, we obtain a subsequence, still denoted by the same index \(n,\) satisfying
\begin{equation}\label{new1-1}
\left(\varphi^{n}, \psi^{n}\right) \rightarrow(\varphi, \psi) \text { in } L^{2}(0, T ; \mathcal{X}), \text { as } n \rightarrow \infty.
\end{equation}
Moreover, the hidden regularity of the solutions given by Proposition \ref{welladj} and the boundedness given by \eqref{normalized_1}
imply that the  boundary terms  $\left\{\varphi^{n}(L, \cdot)\right\}_{n \in \mathbb{N}}$ and \(\left\{\psi^{n}(L, \cdot)\right\}_{n \in \mathbb{N}}\) are bounded in \(H^{\frac{1}{3}}(0, T) .\) Then, the compact embedding
$$
H^{\frac{1}{3}}(0, T) \hookrightarrow L^{2}(0, T)
$$
guarantees that the sequences above are relatively compact in $L^2(0,T)$. Hence, we obtain a subsequence, still denoted by the same index \(n,\) satisfying
\begin{equation}\label{new2}
 \varphi^n(t,L) +a\psi^n(t,L)  \rightarrow  \varphi(t,L) +a\psi(t,L) \quad  \text{in $L^{2}(0,T)$ as $n\rightarrow \infty$.}\\
\end{equation}
In particular,
\begin{equation}\label{new2-1}
 \psi^n(t,L)  \rightarrow  \psi(t,L) \quad  \text{in $L^{2}(0,T)$ as $n\rightarrow \infty$.}\\
\end{equation}
Then, from \eqref{new1} and \eqref{new2}, we deduce that
\begin{equation}\label{new2-1-1}
 \varphi(t,L) +a\psi(t,L) =   0.
\end{equation}
In order to prove that $\{(\varphi^n_T,\psi^n_T)\}_{n\in \N}$ is a Cauchy sequence in $X$, we  use inequality \eqref{putaquepario}:
\begin{multline}
\|(\varphi_T^n, \psi_T^n)\|_{\mathcal{X}}^2  \leq    \frac{1}{T}\|(\varphi^n,\psi^n)\|^2_{L^2(0,T;\mathcal{X})}  +  \frac{1}{c}\int_0^T   (\psi^n)^2(t,L)dt  \\
+ \frac{b}{c}  \int_0^T  \varphi_x^n(t,L) \left(  \varphi_x^n(t,L) + a\psi_x^n(t,L) \right)  dt +   \frac{1}{c}  \int_0^T  \psi_x^n(t,L) \left( ab \varphi_x^n(t,L) + \psi_x^n(t,L) \right)  dt.\nonumber
\end{multline}
From the hidden regularity given in Proposition \ref{welladj} and by using Young inequality, we obtain a constant $C>0$, such that
\begin{multline}
\|(\varphi_T^n, \psi_T^n)\|_{\mathcal{X}}^2  \leq  C\left( \|(\varphi^n,\psi^n)\|^2_{L^2(0,T;\mathcal{X})}   +  \|  \psi^n(\cdot, L) \|_{L^2(0,T)}^2  \right. \\
\left.+  \| \varphi_x^n(\cdot, L) + a\psi_x^n(\cdot, L) \|_{L^2(0,T)}^2 +  \| ab\varphi_x^n(\cdot, L) + \psi_x^n(\cdot, L) \|_{L^2(0,T)}^2 \right).\nonumber
\end{multline}
Thus, the convergences \eqref{new1}, \eqref{new1-1} and \eqref{new2-1} guarantee that $\{(\varphi_T^n,\psi_T^n)\}_{n\in \N}$
is a Cauchy sequence in $\mathcal{X}$. If we denote by $(\varphi_T,\psi_T)$ its limit, from \eqref{normalized_1} we have
$$||(\varphi_T,\psi_T)||_{\mathcal{X}}=1.$$
Moreover, from Proposition \ref{welladj} it follows that 
\begin{align*}
\text{$\{\varphi^n_x(t,L) +a \psi^n_x(t,L)\}_{n\in \N}$, $ \{ab\varphi^n_x(t,L) + \psi^n_x(t,L)\}_{n\in \N}$ and $\{\varphi^n_{xx}(t,0) + a\psi^n_{xx}(t,0)\}_{n\in \N}$}
\end{align*}
 are Cauchy sequences in  $L^2(0,T)$ and $H^{-\frac13}(0,T)$, respectively. Consequently,
\begin{equation*}
\begin{cases}
\varphi^n_x(t,L) +a \psi^n_x(t,L)\rightarrow \varphi_x(t,L) +a \psi_x(t,L), & \text{in $L^2(0,T)$,} \\
ab\varphi^n_x(t,L) +\psi^n_x(t,L)\rightarrow \varphi_x(t,L) +a \psi_x(t,L), & \text{in $L^2(0,T)$,} \\
\varphi^n_{xx}(t,0) + a\psi^n_{xx}(t,0)   \rightarrow \varphi_{xx}(t,0) + a\psi_{xx}(t,0), & \text{in $H^{-\frac13}(0,T)$,}
\end{cases}
\end{equation*}
and \eqref{new1} implies that
\begin{equation}\label{new2-1-2}
 \varphi_x(t,L) +a \psi_x(t,L) =ab \varphi_x(t,L) + \psi_x(t,L) =  \varphi_{xx}(t,0) + a\psi_{xx}(t,0) = 0.
\end{equation}
Finally, from \eqref{new2-1-1} and \eqref{new2-1-2} we obtain that \((\varphi, \psi)\) is a solution of
\begin{equation*}
\begin{cases}
\varphi_t + \varphi_{xxx} + a\psi_{xxx}  =0,  & \text{in $(0,L) \times (0,T)$,} \\
c \psi_t + ab\varphi_{xxx} + r \psi_{x} + \psi_{xxx} =0, & \text{in $(0,L) \times (0,T)$,} \\
\varphi(t,0) =\varphi_x(t,0) =\psi(t,0) =\psi_{x}(t,0) =0, & \text{in $(0,T)$,} \\
 \varphi_{xx}(t,L)+a \psi_{xx}(t,L)=0, & \text{in $ (0,T)$,} \\
 a b\varphi_{xx}(t,L)+{r} \psi(t,L)+ \psi_{xx}(t,L)=0, & \text{in $(0,T)$,} \\
 \varphi(T,x)= \varphi_T(x) \quad \text{and} \quad \psi(T,x)  = \psi_T(x), &  \text{in $(0,L)$,}
\end{cases}
\end{equation*}
with additional boundary conditions
\begin{equation}\label{new4}
\begin{cases}
 \varphi(t,L) +a\psi(t,L) = 0, & \text{in $(0,T)$ } \\
\varphi_x(t,L) +a \psi_x(t,L)= 0, & \text{in $(0,T)$,} \\
ab\varphi_x(t,L) + \psi_x(t,L)= 0, & \text{in $(0,T)$,} \\
\varphi_{xx}(t,0) + a\psi_{xx}(t,0) =0, & \text{in $(0,T)$,}
\end{cases}
\end{equation}
and
\begin{equation*}
||(\varphi_T,\psi_T)||_{\mathcal{X}}=1.
\end{equation*}
Notice that the solutions cannot be identically zero. However, as it will be shown in the following lemmas, one can conclude that $(\varphi,\psi) = (0, 0)$, which drives us to a contradiction.
\end{proof}
\begin{lemma}\label{lemma1}
For any $T>0$, let us denote by $N_{T}$ the space of final state
$\{(\varphi_{T},\psi_{T})\in (L^{2}(0,L))^{2}, $\, \text{such that the mild solution of} \eqref{kdvadjoin}-\eqref{kdvadjcondition}
satisfies the additional boundary condition \eqref{new4}\}.  Then, for $L\in (0,+\infty)$,  $N_{T}=\{ (0,0) \}$.
\end{lemma}
\begin{proof}
This result follows directly from  [Lemma 3.4, \cite{rosier}] and the following lemma.
\end{proof}
\begin{lemma}\label{ucp1}
Let $L\in (0,+\infty)$. Then, does not exist  $\lambda\in\C$ and  $\phi,\vartheta\in H^{3}(0,L)\setminus \{0\}$, such that
\begin{equation}\label{auxsys_1}
\begin{cases}
\lambda \phi  + \phi _{xxx} + a\vartheta_{xxx}=0,  \\
\lambda c \vartheta +r\vartheta_x+ab\phi _{xxx} +\vartheta_{xxx} =0,
\end{cases}
\end{equation}
satisfying  the boundary conditions
\begin{equation}\label{kdvadjconditionr_1}
\begin{cases}
\phi(0) =\phi_x(0)=0,\\
\vartheta(0) =\vartheta_{x}(0)  =0,\\
 \phi_{xx}(L)+a \vartheta_{xx}(L)=0, \\
 a b\phi_{xx}(L)+r \vartheta(L)+ \vartheta_{xx}(L)=0, \\
  \phi(L) +a\vartheta(L) = 0, \\
\phi_x(L) +a \vartheta_x(L)= 0,  \\
ab\phi_x(L) +\vartheta_x(L)= 0,  \\
\phi_{xx}(0) + a\vartheta_{xx}(0)   =0.
\end{cases}
\end{equation}
\end{lemma}
\begin{proof}
 We follow the ideas used in the proof of [Lemma 3.5, \cite{rosier}].  Arguing by contradiction, we  suppose that there exists  $\varphi, \vartheta$ in $H^{3}(0,L)\setminus \{0\}$ solution of  \eqref{auxsys_1}-\eqref{kdvadjconditionr_1} and denote \(\hat{\varphi}(\xi)=\int_{0}^{L} \mathrm{e}^{-i x \xi} \varphi(x) \mathrm{d} x\) and \(\hat{\vartheta}(\xi)=\int_{0}^{L} \mathrm{e}^{-i x \xi} \vartheta(x) \mathrm{d} x\). Then, multiplying the equations in \eqref{auxsys_1} by $e^{-i x \xi}$, integrating by parts over \((0, L)\) and using the boundary conditions \eqref{kdvadjconditionr_1}, we have
\begin{equation*}
\begin{cases}
(\lambda - i \xi^{3}) \hat{\phi} -a i \xi^{3} \hat{\vartheta} =   0, \\
(\lambda c +r i \xi - i \xi^{3} ) \hat{{\vartheta}}  -  abi \xi^{3} \hat{{\phi}} =  \xi^2  e^{-iL\xi}   \left( ab \phi(L)+\vartheta(L) \right)  + \left( ab \phi_{xx}(0)+\vartheta_{xx}(0)\right).
\end{cases}
\end{equation*}
Solving for $\phi$ and $\xi$, when $\lambda= i p$ for $p\in \C$, we obtain
\begin{equation*}
\hat{\phi}(\xi)= \frac{    ia \xi^3 A(\xi)    }{Q(\xi) }		\quad  \text{and} \quad  \hat{\vartheta}(\xi)= \frac{    i(p- \xi^3) A(\xi) }{ Q(\xi) },
\end{equation*}
with
\begin{gather*}
A(\xi) =  -\xi^2 e^{-iL\xi} \left( ab \phi(L)+\vartheta(L) \right)   -  \left( ab \phi_{xx}(0)+\vartheta_{xx}(0) \right)
\end{gather*}
and
\begin{gather}\label{poly}
Q(\xi) =\left(1-a^2 b\right)\xi ^6 - r\xi ^4 -(c+1)  p \xi ^3+pr \xi  +c p^2.
\end{gather}
If we use the  characterization of  Fourier transform for function  in $H^{2}(\R)$ and  Paley-Wiener Theorem [\cite{Yosida}, page 161], we deduce that the existence of nontrivial solutions of the problem \eqref{auxsys_1}-\eqref{kdvadjconditionr_1} is equivalent to the existence of $p\in \C$ and $\alpha,\beta, \gamma \in\C \setminus\{0\}$, such that
\begin{enumerate}
\item $\upsilon_j$ for $j=1,2$  are entire functions,
\item $\int_{\R} |\upsilon_j(\xi)|^{2}(1+|\xi|^{2})^{2}d\xi < \infty$ for $j=1,2$,
\item $ |\upsilon_j(\xi)|\le C(1+|\xi|)^{N} e^{L |Im(\xi)|}$ for $j=1,2$ and  for all  $\xi \in \C$,
\end{enumerate}
where $\upsilon_j$, $j=1,2$,  are given by
\begin{equation}\label{entirefour1}
\upsilon_1(\xi)= \frac{    ia \xi^3 \left(  	\beta \xi^2 e^{-iL\xi}     - \gamma \right)  }{Q(\xi) }		\quad  \text{and} \quad  \upsilon_2(\xi)= \frac{    i(p- \xi^3) \left( \beta \xi^2 e^{-iL\xi}    - \gamma  	\right) }{ Q(\xi) },
\end{equation}
with  $\beta = - \left( ab \phi(L)+\vartheta(L) \right)$ and $\gamma= ab \phi_{xx}(0)+\vartheta_{xx}(0)$ from \eqref{auxsys_1} are nonzero real numbers.
On the other hand, note that if $\upsilon_1(\xi)$ and $ \upsilon_2(\xi)$ are entire functions,  then the functions given by
\begin{equation}
\begin{aligned}\label{entirefuncion4}
\mu_1(\xi)&= (1-a^2b)\upsilon_1(-\xi)=\frac{    -ia \xi^3 \left(  	\beta \xi^2 e^{iL\xi}     - \gamma \right)  }{P(\xi) },		\\
\mu_2(\xi)&= (1-a^2b)\upsilon_2(-\xi)= \frac{    i(p+ \xi^3) \left( \beta \xi^2 e^{iL\xi}    - \gamma  	\right) }{ P(\xi) },
\end{aligned}
\end{equation}
where 
\begin{equation*}
P(\xi)=\frac{Q(-\xi)}{(1-a^2b)}=\xi ^6 - \frac{r}{\left(1-a^2 b\right)}\xi ^4 +\frac{(c+1)  p}{\left(1-a^2 b\right)} \xi ^3-\frac{pr}{\left(1-a^2 b\right)} \xi  +\frac{c p^2}{\left(1-a^2 b\right)},
\end{equation*}
 are entire functions as well. Moreover, $\mu_1$ and $\mu_2$ are entire functions if and only if the roots of \(P(\xi)\) are also roots of $ ia \xi^3 \left(  	\beta \xi^2 e^{iL\xi}     - \gamma \right)$ and  $i(p+ \xi^3)\left(  	\beta \xi^2 e^{iL\xi}     - \gamma \right)$.  Hence, if we consider  \(\xi_{0}, \xi_{1}, \xi_{2}, \xi_{3}, \xi_{4}\) and \(\xi_{5}\) the roots of \(P(\xi)\). The polynomial $P(\cdot)$ can be written as
\begin{align*}
{P(\xi)}=  \prod_{k=0}^{5} (\xi-\xi_{k}) =  \xi^6 +A \xi^5 +B \xi^4 +C \xi^3 +D\xi^2 +E\xi + F 
\end{align*}	
whit
\begin{align*}
-A&=  \sum_{k=0}^{5}\xi_{k}, \\
B&= \xi_0 \sum_{k=1}^{5}\xi_{k} +   \xi_1  \sum_{k=2}^{5}\xi_{k}+\xi_2 \sum_{k=3}^{5}\xi_{k}+\xi_3 \sum_{k=4}^{5}\xi_k+\xi_4\xi_5, \\
-C&=\xi_0\xi_1   \sum_{k=2}^{5}\xi_{k}+\xi_4\xi_5 \sum_{k=0}^{3}\xi_{k}  +\xi_2\xi_3 \sum_{k=0, k\neq 2, k\neq 3}^{5}\xi_{k} + \sum_{k=0}^{1}\xi_{k}\sum_{i=2}^{3}\xi_{i}\sum_{j=4}^{5}\xi_{j}, \\
D&=\xi_0\xi_1 \sum_{k=2}^{3}\xi_{k}\ \sum_{J=4}^{5}\xi_{j}+\xi_2\xi_3 \sum_{k=0}^{1}\xi_{k}\ \sum_{J=4}^{5}\xi_{j}+ \xi_4\xi_5 \sum_{k=0}^{1}\xi_{k}\ \sum_{J=2}^{3}\xi_{j} + \prod_{k=0}^{3}\xi_k+  \prod_{k=2}^{5}\xi_k +  \prod_{k=0, k\neq 2, k\neq 3}^{5}\xi_k ,\\
-E&=\prod_{j=0}^{3}\xi_j\sum_{k=4}^{5}\xi_{k} +\prod_{j=2}^{5}\xi_j\sum_{k=0}^{1}\xi_{k}+\prod_{j=0, j\neq 2, j\neq 3}^{5}\xi_j \sum_{k=2}^{3}\xi_{k},\\
F&=\prod_{k=0}^5\xi_k.
\end{align*}
By using the Newton-Girard relations, we obtain that 
\begin{align}\label{girard}
A= 0, \quad B=-\frac{r}{1-a^2b}, \quad C=\frac{(c+1)  p}{1-a^2b}, \quad D=0, \quad E=-\frac{ pr}{1-a^2b} 	\quad \text{and} \quad F= \frac{c  p^2}{1-a^2b}.
\end{align}
Note that every root of ${P(\xi)}$ is also root of 
\begin{equation}\label{laultima}
\mathcal{B}(z)=\beta z^2 e^{iLz}    - \gamma.
\end{equation}
In this case, we have that $\mathcal{B}(\xi_j)=0$, for all $j 	\in 	\{0,1,2,3,4,5\}$, i. e.,
\begin{equation}\label{lambert}
 \xi_j^2e^{iL \xi_j}=	\frac{\gamma}{\beta} ,   \quad \forall j 	\in \{0,1,2,3,4,5\},
\end{equation}
which implies that
\begin{align*}
|\xi_j|^2=\left| \frac{\gamma}{\beta}\right| e^{L\mathcal{I}m \xi_j},   \quad \forall j 	\in \{0,1,2,3,4,5\} \quad \text{and} \quad \gamma,  \beta \in \R\setminus\{0\}.
\end{align*}
To obtain a contradiction, we proceed in several steps. More precisely,  we will consider the following cases: $p \in  \C \setminus \{\R\cup i\R\}$, $p \in \R\setminus \{0\}$, $p \in i\R\setminus\{0\}$ and $p=0$ as follow: 
\begin{itemize}
\item\textbf{ Case 1}: Initially, suppose that \underline{${p \in \C \setminus \{\R\cup i\R\}}.$}\\
Observe that, in this case, zero is not the root of $P(\xi)$. Moreover,  from \eqref{lambert}, we deduce that each $\frac{iL}{2}\xi_j$ is a solution of the equation  
\begin{equation}\label{trascendental}
ze^z=\alpha,
\end{equation}
where $\alpha= \pm \frac{iL}{2}\left(\frac{\gamma}{\beta}\right)^{1/2}.$ This peculiar transcendental equation has been studied by a large number of authors (see, for instance, \cite{corless}, \cite{siewert}). The solutions are given by the Lambert function $W (z)$, i.e. 
$$W(z)e^{W(z)}=\alpha.$$
Moreover, if $\alpha \neq 0$, the equation \eqref{trascendental} has infinitely many solutions. Indeed, the solutions $z_k$ of the equation \eqref{trascendental} are given by the zeros of the holomorphic complex function 
\begin{equation}\label{W}
W_k(z)=\log (\alpha) - z - \log(z) + 2k \pi i, 	\quad k\in \Z, 
\end{equation}
where $\log (z)$ denotes the principal branch of the log-function. Thus, from \eqref{lambert}, we deduce that there exist $k_j \in \Z$, for each $j \in \{0,1,2,3,4,5\}$ with $k_j \neq k_i$, such that 
\begin{equation*}
W_{j}\left(\frac{iL}{2}\xi_j\right)=0,
\end{equation*}
i.e.,
\begin{equation*}
\log (\alpha) - \frac{iL}{2}\xi_j - \log\left| \frac{L}{2} \xi_j\right| - i \arg \left(\frac{iL}{2}\xi_j\right) + 2k_j \pi i=0.
\end{equation*}
Here, we consider the case $\alpha = \frac{iL}{2}\left(\frac{\gamma}{\beta}\right)^{1/2}$, the case $\alpha = -\frac{iL}{2}\left(\frac{\gamma}{\beta}\right)^{1/2}$ is similar and we omit it. Thus, observe that 
\begin{equation}\label{argumentalpha}
\log (\alpha)= \begin{cases}
\log\left( \frac{L}{2} \left(\frac{\gamma}{\beta}\right)^{1/2}\right) + i \left(\frac{\pi}{2}+2k\pi\right), & \text{if $\frac{\gamma}{\beta}>0$, $k \in\Z$,} \\
\log\left( \frac{L}{2} \left(-\frac{\gamma}{\beta}\right)^{1/2}\right) + i\left( \pi + 2k\pi\right), & \text{if $\frac{\gamma}{\beta}<0$,  $k \in\Z$.}
\end{cases}
\end{equation}
Then, from relations \eqref{girard}, we obtain 
\begin{equation*}
6\log (\alpha)  - \sum_{j=0}^5 \log\left| \frac{L}{2} \xi_j\right|- i\sum_{j=0}^5 \arg \left(\frac{iL}{2}\xi_j\right) +2m\pi i=0,
\end{equation*}
where $m=\sum_{j=0}^5 k_j  \in \Z$. Consequently, 
\begin{equation*}
6\log (\alpha)  - 6 \log\left( \frac{L}{2}\right) -   \sum_{j=0}^5\log\left|\xi_j\right| - i\sum_{j=0}^5 \arg \left(\frac{iL}{2}\xi_j\right) +2m\pi i=0.
\end{equation*}
Since $A=0$ in \eqref{girard}
\begin{align*}
\sum_{j=0}^5\log\left|\xi_j\right|&=\sum_{j=0}^5\frac12\log\left|\xi_j\right|^2= \frac12\sum_{j=0}^5\log\left(\left| \frac{\gamma}{\beta}\right| e^{L\mathcal{I}m \xi_j}\right) \\
&= \frac12\sum_{j=0}^5\log\left| \frac{\gamma}{\beta}\right| +\frac{L}{2}\sum_{j=0}^5\mathcal{I} m \xi_j= 3\log\left| \frac{\gamma}{\beta}\right|,
\end{align*}
we get from the definition of $\log (\alpha)$ in \eqref{argumentalpha} if $\frac{\gamma}{\beta}>0$ then
\begin{equation*}
3\log |\alpha|^2+6 \arg (\alpha)  - 6 \log\left( \frac{L}{2}\right) -  3\log\left| \frac{\gamma}{\beta}\right| - i\sum_{j=0}^5 \arg \left(\frac{iL}{2}\xi_j\right) +2m\pi i=0,
\end{equation*}
or if $\frac{\gamma}{\beta}<0$ so
\begin{equation*}
3\log\left( \left(\frac{L}{2}\right)^2\left|\frac{\gamma}{\beta}\right|\right)+6 \arg (\alpha)  - 6 \log\left( \frac{L}{2}\right) -  3\log\left| \frac{\gamma}{\beta}\right|- i\sum_{j=0}^5 \arg \left(\frac{iL}{2}\xi_j\right) +2m\pi i=0.
\end{equation*}
Hence,
\begin{equation*}
6\arg (\alpha) - i\sum_{j=0}^5 \arg \left(\frac{iL}{2}\xi_j\right) +2m\pi i=0.
\end{equation*}
Finally, \eqref{argumentalpha} allows us to conclude that
\begin{equation}\label{Falpha}
 \sum_{j=0}^5 \arg \left(\frac{iL}{2}\xi_j\right)  = F(\alpha), \quad \text{where}\quad F(\alpha) = \begin{cases}
(2m+12k +3)\pi,  & \text{if $\frac{\gamma}{\beta}>0$, $k\in \Z$,} \\
2(m+3+6k)\pi, & \text{if $\frac{\gamma}{\beta}<0$, $k\in \Z$}.
\end{cases}
\end{equation}
On the other hand, $$\prod_{j=0}^{5}\left(\frac{iL}{2}\xi_j\right)=\prod_{j=0}^{5}\frac{L}{2}|\xi_j| e^{i \arg\left(i\frac{L}{2}\xi_j\right)}.$$
 Therefore, from \eqref{girard}  we have
\begin{equation*}
\prod_{k=0}^5\xi_k= \frac{cp^2}{1-a^2b} \quad \text{and} \quad \prod_{k=0}^5|\xi_k|= \frac{c|p|^2}{1-a^2b}.
\end{equation*}
Combining the above identities, we obtain the following relation
\begin{align*}
 \left( \frac{iL}{2}\right)^6 \frac{cp^2}{1-a^2b} = \left( \frac{L}{2}\right)^6 \frac{c|p|^2}{1-a^2b} e^{i \sum_{k=0}^5 \arg\left(i\frac{L}{2}\xi_j\right)} 
\end{align*}
which implies that
\begin{align*}
-  \frac{p^2}{ |p|^2} = \cos(F(\alpha)) + i \sin(F(\alpha)).
\end{align*}
Finally, from \eqref{Falpha} it follows that the coefficient $p$ for the polynomial $P(\xi)$ satisfies
\begin{equation}\label{eqnew2}
   \frac{p^2}{ |p|^2} = \begin{cases}
 1, & \text{if $\frac{\gamma}{\beta}>0$}, \\
-1,& \text{if $\frac{\gamma}{\beta}<0$}.
\end{cases} 
\end{equation}
Moreover, it implies that $\mathcal{I}m \left( \frac{p^2}{ |p|^2} \right) = 0$. Hence $\frac{2 \mathcal{R}e (p) \mathcal{I} m (p)}{\mathcal{R}e (p)^2 + \mathcal{I}m(p)^2}=0$, and therefore, we deduce that $$ \mathcal{R}e (p) =0 \quad \text{or} \quad  \mathcal{I}m(p)=0,  $$
which drives us to a contradiction since $p \in \C\setminus (\R \cup i\R)$. \\

\item \emph{\textbf{Case 2:}} Suppose that \underline{${p \in \R\setminus \{0\}}$}.\\

Then, if $P(\cdot)$ has a root $\xi_0 \in \C\setminus \R$, then $\bar{\xi}_0$ also is root. 
Then, there exist integers $k$ and $m$, such that 
\begin{align*}
W_k(\xi_0)=W_m(\bar{\xi}_0)=0,
\end{align*}
equivalently, from the definition of $W_k, W_m$ in \eqref{W}
\begin{equation*}
\begin{cases}
\log (|\alpha|) + i\arg (\alpha) - \xi_0 - \log(|\xi_0|)- i\arg(\xi_o) + 2k \pi i=0  \\
\log (|\alpha|) +i\arg (\alpha)- \bar{\xi}_0 - \log(|\bar{\xi}_0|)  - i\arg(\bar{\xi}_o)+ 2m \pi i=0 
\end{cases} 
\end{equation*}
and applying the conjugate function to both equations 
\begin{align*}
\begin{cases}
\log (|\alpha|) - i\arg (\alpha) - \bar{\xi_0} - \log(|\xi_0|)+ i\arg(\xi_o) - 2k \pi i=0  \\
-\log (|\alpha|) -i\arg (\alpha)+ \bar{\xi}_0 + \log(|\bar{\xi}_0|)  - i\arg(\bar{\xi}_o)-2m \pi i=0.
\end{cases} 
\end{align*}
Adding the above equations, we obtain $\arg (\alpha)= - (k+m-1)\pi$. From \eqref{argumentalpha}, we deduce that 
\begin{align*}
 - (k+m-1)\pi=
\begin{cases}
\frac{\pi}{2} & \text{if $\frac{\gamma}{\beta}>0$}, \\
\pi  & \text{if $\frac{\gamma}{\beta}<0$}. 
\end{cases}
\end{align*}
If $\frac{\gamma}{\beta}>0$, we have a contradiction. On the other hand, if $\frac{\gamma}{\beta}<0$, from \eqref{eqnew2} we deduce that $\frac{p^2}{|p|^2}=-1$, but since $p \in \R$, this is a contradiction as well. Hence, we have that  all roots of  $P(\cdot)$ are real, so from \eqref{lambert}, it follows that $L\xi_j=n_j\pi$ for some $n_j \in \Z$ and $j \in \{0,1,2,3,4,5\}$, then
\begin{equation}\label{auxnorm}
\xi^2_j =(-1)^{n_j} \left(\frac{\gamma}{\beta}\right)  \quad \text{and} \quad  |\xi_j| = \left|\frac{\gamma}{\beta}\right|^{\frac12}.
\end{equation}
Let $J=\{j\in \N : \xi_j>0\}$ and $K=\{k\in \N : \xi_k<0\}$, notice that zero never is a root and $J \cup K=\{0,1,2,3,4,5\}$. If $J=\emptyset$ or $K=\emptyset$, then 
\begin{equation*}
\sum_{i=0}^{5}\xi_0 > 0  \quad \text{or} \quad \sum_{i=0}^{5}\xi_0 < 0,
\end{equation*}
which  contradicts \eqref{girard}. Thus,  $J\neq\emptyset$ and $K\neq\emptyset$.\\
Due \eqref{girard},  \eqref{auxnorm}, since all the roots are real numbers and $\sum_{j\in J} \xi_j + \sum_{k\in K} \xi_k=0$, we conclude that the cardinality of $J$ and $K$ is 3. Moreover, 
From \eqref{auxnorm}, for all $j\in J$ and $k\in K$,
 $|\xi_j| -|\xi_k|= \xi_j + \xi_k=0$,  and using \eqref{girard}, without loss of generality orderings 
$J=\{ \xi_4,\xi_0,\xi_2\}$ and $K=\{\xi_5,\xi_1,\xi_3\}$,
by \eqref{girard}, it follows that 
\begin{equation*}
\prod_{j=0}^{3}\xi_j(\xi_{4}+\xi_5) +\prod_{j=2}^{5}\xi_j(\xi_{0}+\xi_1)+\prod_{j=0, j\neq 2, j\neq 3}^{5}\xi_j (\xi_{2}+\xi_3)= 0= \frac{-pr}{1-a^2b},
\end{equation*}
driving to the conclusion that $p=0$,  which is a contradiction.\\ 

\item \emph{\textbf{Case 3:}} Suppose that \underline{${p=iq\quad \text{and} \quad	q \in \R\setminus \{0\}}$}. In this case, note that the function $f(\xi)=(1-a^2b)\upsilon_1(i\xi),$ for $v_1$ defined by \eqref{entirefour1}, is entire and it is given by 
\begin{equation*}
f(\xi)= \frac{  (1-a^2b)  a \xi^3 \left(  	-\beta \xi^2 e^{L\xi}     - \gamma \right)  }{Q(i\xi) }	 = \frac{ -   a \xi^3 \left(  	\beta \xi^2 e^{L\xi}     + \gamma \right)  }{R(\xi)}, 
\end{equation*}
where the polynomial $R(\cdot)$ has real coefficients and it is given by 
\begin{equation*}
R(\xi)=\xi ^6 -\frac{ r}{\left(1-a^2 b\right)}\xi ^4 -\frac{(c+1)  q}{\left(1-a^2 b\right)} \xi ^3+\frac{qr}{\left(1-a^2 b\right)} \xi  +\frac{c q^2}{\left(1-a^2 b\right)}.
\end{equation*}
for $j\in \{0,1,2,3,4,5\}$ let us to consider $\xi_j$ to be the roots of $R(\cdot)$. Proceeding as in case 2, we also obtain that all the roots of $R(\cdot)$ are real numbers. Moreover, since the polynomial $R(\cdot)$ does not have a term of degree five, the Newton-Girard relations also imply that
\begin{equation}\label{girard2}
\sum_{j=0}^5\xi_j=0.
\end{equation}
Since $ \xi_j^2 e^{L\xi_j}  =-\frac{ \gamma}{\beta}$, for all $j\in \{0,1,2,3,4,5\}$, implying that $\frac{ \gamma}{\beta} <0$ and 
\begin{equation*}
 \xi_j = \xi_k e^{\frac{L}{2}(\xi_k-\xi_j)}, \quad \forall j,k \in\{0,1,2,3,4,5\}.
\end{equation*}
Hence, from \eqref{girard2}, it follows that 
\begin{align*}
0=\xi_0 + \sum_{j=1}^5\xi_j= \xi_0 + \sum_{j=1}^5\xi_0e^{\frac{L}{2}(\xi_0-\xi_j)}= \xi_0\left(1 +  \sum_{j=1}^5e^{\frac{L}{2}(\xi_0-\xi_j)}\right),
\end{align*}
which drives us to the contradiction.\\

\item \textbf{Case 4:} Suppose that \underline{${p=0.}$} Thus, we have that $\xi_0=0$ is a root of $P(\xi)$ and the polynomial $P$ can be written as
\begin{equation}
P(\xi)=\frac{\xi^4 \left( \left(1-a^2 b\right) \xi ^2 - r\right)}{(1-a^2b)}.\nonumber
\end{equation}
Hence,  from \eqref{entirefuncion4}, the rational function
\begin{equation*}
 \mu_1(\xi)=\frac{    -ia  \left(  \beta \xi^2 e^{iL\xi}	    - \gamma \right)  }{\xi\left(\left(1-a^2 b\right)  \xi ^4 - r\right)}
\end{equation*}
is entire if zero is a root of   $\beta \xi^2e^{iL\xi}  - \gamma  $, which contradicts the fact that $\gamma$ is a nonzero real number.
\end{itemize}
In any case, we have a contradiction. Thus, the lemma holds. This completes the proof of Lemma \ref{lemma1} and, consequently, the proof of Proposition \ref{prop1}.
\end{proof}
\subsubsection{\textbf{Further Configurations of Four Controls}}
Following the steps of the proof of Proposition \ref{prop1}, we obtain similar results for the following configuration with four controls:
\begin{equation}\label{another1}
\left\lbrace
\begin{array}{llll}
u(t,0)=0, & u_x(t,L)=h_1(t), & u_{xx}(t,L)=0 & \text{in $(0,T)$},\\
v(t,0)=g_0(t), & v_x(t,L)=g_1(t), & v_{xx}(t,L)=g_2(t), & \text{in $(0,T)$},
\end{array}\right.
\end{equation}
\begin{equation}\label{another2}
\left\lbrace
\begin{array}{llll}
u(t,0)=h_0(t), & u_x(t,L)=h_1(t), & u_{xx}(t,L)=0 & \text{in $(0,T)$},\\
v(t,0)=g_0(t), & v_x(t,L)=g_1(t), & v_{xx}(t,L)=0, & \text{in $(0,T)$},
\end{array}\right.
\end{equation}
and
\begin{equation}\label{another3}
\left\lbrace
\begin{array}{llll}
u(t,0)=0, & u_x(t,L)=h_1(t), & u_{xx}(t,L)=h_2(t) & \text{in $(0,T)$},\\
v(t,0)=0, & v_x(t,L)=g_1(t), & v_{xx}(t,L)=h_3(t), & \text{in $(0,T)$}.
\end{array}\right.
\end{equation}
Indeed, with the \textbf{controls configuration  \eqref{another1}},  the observability inequality associated with the control problem is given by
\begin{multline}\label{desig1}
\| (\varphi_T,\psi_T)\|^2_{\mathcal{X}} \le C \left( \|  ab\varphi(t,L) + \psi(t,L)\|^{2}_{H^{\frac13}(0,T)}
\right. \\
+ \|  \varphi_x(t,L) + a\psi_x(t,L) \|^{2}_{L^2(0,T)}   +  \|  ab \varphi_x(t,L) + \psi_x(t,L) \|^{2}_{L^2(0,T)}  \\
\left. \| ab \varphi_{xx}(t,0) + \psi_{xx}(t,0)  \|^{2}_{H^{-\frac13}(0,T)}\ \right).
\end{multline}
Arguing by contradiction and proceeding as in the proof of Proposition \ref{prop1}, the problem is reduced to proving a unique continuation property for the solutions of a spectral system.
In this case, the Fourier and Paley-Wiener approach leads to the study of the following entire functions
\begin{equation*}
\upsilon(\xi)= \frac{   i(p+r\xi-  \xi^3) \left(  \beta \xi^2 e^{-iL\xi}		    - \gamma \right)  }{P(\xi) }		
\quad  \text{and} \quad  \nu(\xi)= \frac{    ia\xi^3 \left( \beta \xi^2 e^{-iL\xi}	   - \gamma  	\right) }{ P(\xi) },
\end{equation*}
whose properties are similar to those of the functions given in \eqref{entirefour1}. Therefore, they
can be analyzed as in the proof of Lemma \ref{ucp1} and, consequently, the controllability property
holds for any $L>0$. \\

For the \textbf{controls configuration  \eqref{another2}}, we can prove the following observability inequality:
\begin{multline}\label{desig2}
\| (\varphi_T,\psi_T)\|^2_{\mathcal{X}}\le C \left( \|   \varphi_{xx}(t,0) +a \psi_{xx}(t,0) \|^2_{H^{-1/3}(0,T)} \right. \\
 + \|  \varphi_x(t,L) + a\psi_x(t,L) \|^{2}_{L^2(0,T)}   + \|  ab\varphi_x(t,L) + \psi_x(t,L) \|^{2}_{L^2(0,T)}   \\
\left.\| ab\varphi_{xx}(t,0) + \psi_{xx}(t,0)\|^2_{H^{-1/3}(0,T)} \right).
\end{multline}
As in the previous cases, a contradiction argument reduces the problem to analyze the unique continuation for the system
\begin{equation}\label{auxsys3}
\begin{cases}
\lambda \varphi + \varphi_{xxx} + a\psi_{xxx}=0,  \\
\lambda \psi +\frac{r}{c}\psi_x+\frac{ ba}{c}\varphi_{xxx} +\frac{1}{c}\psi_{xxx} =0,
\end{cases}
\end{equation}
satisfying  the boundary conditions
\begin{gather}\label{kdvadjcondition3}
\begin{cases}
\varphi(0) =\varphi_x(0)= \varphi_{xx}(0)  =\varphi_{x}(L)=0,\\
\psi(0) =\psi_{x}(0) = \psi_{xx}(0) =\psi_{x}(L)=0,\\
 \varphi_{xx}(L)+a \psi_{xx}(L)=0, \\
 a b\varphi_{xx}(L)+r \psi(L)+ \psi_{xx}(L)=0.
\end{cases}
\end{gather}
We claim that, for all $L\in \R^{+}$,   the  unique solution of \eqref{auxsys3}-\eqref{kdvadjcondition3} is $(\varphi,\psi)\equiv ( 0,0)$. Indeed,
we suppose that there exists  $\varphi, \psi$ in $H^{3}(0,L)\setminus \{0\}$, solution of  \eqref{auxsys3}-\eqref{kdvadjcondition3}, and let us denote
\(\hat{\varphi}(\xi)=\int_{0}^{L} \mathrm{e}^{-i x \xi} \varphi(x) \mathrm{d} x\) and \(\hat{\psi}(\xi)=\int_{0}^{L} \mathrm{e}^{-i x \xi} \psi(x) \mathrm{d} x\).
Then, multiplying the equations in \eqref{auxsys3} by e \(^{-i x \xi},\) integrating by parts over \((0, L)\) and
setting $\lambda= i p$, with $p\in \C$, from the Paley-Wiener Theorem
it follows that $\hat{\varphi}$ and $\hat{\psi}$ are entire functions given  by
\begin{equation}\label{entirefour2}
\hat{\varphi}(\xi)=  e^{-iL\xi} \frac{ B(\xi) }{P(\xi) }		\quad  \text{and} \quad  \hat{\psi}(\xi)= e^{-iL\xi} \frac{ C(\xi) }{P(\xi) }	,
\end{equation}
with
\begin{gather*}
P(\xi) =\left(1-a^2 b\right)\xi ^6 - r\xi ^4 -(c+1)  p \xi ^63+pr \xi  +c p^2,\\
B(\xi) =  ia \alpha_1  \xi^5    +i\alpha_2\xi^2(pc+r\xi -\xi^3),\\
C(\xi) =  iab \alpha_2\xi^5     +i\alpha_1\xi^2(p -\xi^3),\nonumber
\end{gather*}
and the coefficients $  \alpha_1=   ab \varphi(L)+\psi(L)$ and $\alpha_2 =  \varphi(L)+a\psi(L).\nonumber$
It drives us to contradiction since $\frac{ B(\xi) }{P(\xi) }$ and $\frac{ C(\xi) }{P(\xi) }$ cannot be entire functions. Since  $B(\xi)$ and $C(\xi)$ are polynomials of degree five, and $P(\xi)$ is a polynomial of degree six. \\

Finally, for the \textbf{controls configuration} \eqref{another3}, the observability inequality associated with the control problem is given by
\begin{multline}\label{desig3}
\| (\varphi_T,\psi_T)\|^2_{\mathcal{X}}   \le C \left(    \|  \varphi(t,L) +a\psi(t,L)\|^{2}_{H^{\frac13}(0,T)}
 + \| ab\varphi(t,L) +\psi(t,L) \|^{2}_{H^{\frac13}(0,T)}\right.  \\
  \left.  + \|  \varphi_x(t,L) +a \psi_x(t,L) \|^{2}_{L^2(0,T)}    +\|  ab \varphi_x(t,L) + \psi_x(t,L) \|^{2}_{L^2(0,T)}  \right).
\end{multline}
The approach applied in the previous cases reduces the problem to prove a unique continuation property for the system
\begin{equation}\label{auxsys4}
\begin{cases}
\lambda \varphi + \varphi_{xxx} + a\psi_{xxx}=0,  \\
\lambda \psi +\frac{r}{c}\psi_x+\frac{ ba}{c}\varphi_{xxx} +\frac{1}{c}\psi_{xxx} =0,\nonumber
\end{cases}
\end{equation}
satisfying  the boundary conditions
\begin{gather}\label{kdvadjcondition4}
\begin{cases}
\varphi(0) =\varphi_x(0)= \varphi(L)= \varphi_x(L) =0,\\
\psi(0) =\psi_{x}(0) = \psi(L)=\psi_x(L) =0,\\
 \varphi_{xx}(L)+a \psi_{xx}(L)=0, \\
 a b\varphi_{xx}(L)+r \psi(L)+ \psi_{xx}(L)=0.\nonumber
\end{cases}
\end{gather}
Then, entire functions associated with the unique continuation property mentioned above are $\hat{\varphi}$ and $\hat{\psi}$,  given  by
\begin{equation*}
\hat{\varphi}(\xi)=  \frac{ -ia \left( ab \varphi_{xx}(0)+\psi_{xx}(0)\right)\xi^3-i\left( \varphi_{xx}(0)+a\psi_{xx}(0)\right)(pc-r\xi -\xi^3)}{\left(1-a^2 b\right)\xi ^6 - r\xi ^4 -(c+1)  p \xi ^63+pr \xi  +c p^2 }	
\end{equation*}
and
\begin{align*}
 \nu(\xi)=\frac{ -iab\left( \varphi_{xx}(0)+a\psi_{xx}(0)\right)\xi^3-i\left( ab\varphi_{xx}(0)+\psi_{xx}(0)\right)(p-\xi^3) }{\left(1-a^2 b\right)\xi ^6 - r\xi ^4 -(c+1)  p \xi ^63+pr \xi  +c p^2}.
\end{align*}
It drives us to a contradiction since the numerators of $\hat{\varphi}$ and $\hat{\psi}$ are polynomials of degree three, and the corresponding denominators are polynomials of degree six. Thus,  $\hat{\varphi}$ and $\hat{\psi}$ cannot be entire functions.
\subsection{Case  B: Observability inequality with three controls configuration}
In this subsection, we consider the following boundary control configuration ($g_0\neq 0, g_1=g_2=0$ and $h_0 \neq 0,   h_1 \neq 0, h_2 = 0$):
\begin{equation*}
\left\lbrace
\begin{array}{llll}
u(t,0)=h_0(t), & u_x(t,L)=h_1(t), & u_{xx}(t,L)=0, & \text{in $(0,T)$},\\
v(t,0)=g_0(t), & v_x(t,L)=0, & v_{xx}(t,L)=0, & \text{in $(0,T).$}
\end{array}\right.
\end{equation*}
As it is well known, the control problem should be more difficult if one removes controls. However, in the case of three controls, we obtain a positive answer for the exact controllability problem by using the approach introduced by Rosier in \cite{rosier}, provided that some relation between the coefficients, the length of the interval, and the time $T$ holds. More precisely, under the above configuration of controls, we establish the following observability inequality:
\begin{proposition}\label{prop2}
  Let $T>0$ and  $L > 0$. Assume that the parameters $a, b, c \in \R$ satisfy
\begin{equation}\label{putaquepario3}
0<  C_1(	1-a^2b) < c,
\end{equation}
  where $C_1$ is the hidden regularity constant given in \eqref{putaquepario2}. Then, there exists  $C=C(L,T)>0$,  such that, if $(\varphi,\psi) \in C([0,T];L^2(0,L))^2 $ is a solution of  \eqref{kdvadjoin}-\eqref{kdvadjcondition},  the inequality
\begin{multline}\label{obser3}
\| \psi_T\|^2_{L^2(0,L)} + \frac{b}{c} \| \varphi_T\|^2_{L^2(0,L)} \le C \left( \|  \varphi_{xx}(t,0) +a \psi_{xx}(t,0) \|^{2}_{H^{-\frac13}(0,T)}
\right. \\
\left. + \|  \varphi_x(t,L) + a\psi_x(t,L) \|^{2}_{L^2(0,T)}   + \| ab \varphi_{xx}(t,0) + \psi_{xx}(t,0)  \|^{2}_{H^{-\frac13}(0,T)}\ \right)
\end{multline}
holds for any $\varphi_T, \psi_T \in L^2(0,L)$.
\end{proposition}
\begin{proof}
We argue by contradiction and following the steps of the proof of Proposition \ref{prop1}. The main difference between them is the absence  of the condition $ ab\varphi^n_x(t,L) + \psi^n_x(t,L)\rightarrow 0$ in $L^2(0,T)$,
as $n\rightarrow \infty$; therefore, we omit some details. So, we obtain a sequence of functions $(\varphi^n_T,\psi^n_T)$, such that
\begin{multline}\label{normalized_1'}
1=\|(\varphi^n_T,\psi^n_T)\|_{\mathcal{X}} > n \left( \|   \varphi^n_{xx}(t,0) +a \psi^n_{xx}(t,0) \|^{2}_{-H^{\frac13}(0,T)} \right. \\
 \left.+ \|  \varphi^n_x(t,L) +a \psi^n_x(t,L) \|^{2}_{L^2(0,T)}    +\|  ab\varphi^n_{xx}(t,0) + \psi^n_{xx}(t,0)  \|^{2}_{H^{-\frac13}(0,T)}\ \right).
\end{multline}
From \eqref{normalized_1'}, we get
\begin{equation}\label{new1'}
\begin{cases}
 \varphi^n_{xx}(t,0) +a\psi^n_{xx}(t,0)  \rightarrow 0, & \text{in $-H^{\frac13}(0,T)$, } \\
\varphi^n_x(t,L) +a \psi^n_x(t,L)\rightarrow 0, & \text{in $L^2(0,T)$,} \\
ab\varphi^n_{xx}(t,0) + \psi^n_{xx}(t,0)   \rightarrow 0, & \text{in $H^{-\frac13}(0,T)$,}
\end{cases}
\end{equation}
as $n\rightarrow \infty$.
In order to prove that $\{(\varphi^n_T,\psi^n_T)\}$ is a Cauchy sequence in $X$, we  use the inequality \eqref{putaquepario} to obtain
\begin{multline}
\|(\varphi_T^n, \psi_T^n)\|_{\mathcal{X}}^2  \leq    \frac{1}{T}\|(\varphi^n,\psi^n)\|^2_{L^2(0,T;\mathcal{X})}  +  \frac{1}{c}\int_0^T   (\psi^n)^2(t,L)dt  \\
+ \frac{b}{c}  \int_0^T  \left(  \varphi_x^n(t,L) + a\psi_x^n(t,L) \right)^2  dt +   \frac{1-a^2b}{c}  \int_0^T  (\psi_x^n)^2(t,L)   dt.\nonumber
\end{multline}
Then, from the hidden regularity given in Proposition \ref{welladj}, we obtain a constant $C_1>0$ satisfying
\begin{multline}
\|(\varphi_T^n, \psi_T^n)\|_{\mathcal{X}}^2  \leq    \frac{1}{T}\|(\varphi^n,\psi^n)\|^2_{L^2(0,T;\mathcal{X})}  +  \frac{1}{c}\int_0^T   (\psi^n)^2(t,L)dt  \\
+ \frac{b}{c}  \int_0^T  \left(  \varphi_x^n(t,L) + a\psi_x^n(t,L) \right)^2  dt +   \left(\frac{1-a^2b}{c}\right) C_1 \int_0^L  (\psi^n_T)^2   dx,\nonumber
\end{multline}
which allows to conclude that
\begin{equation}\label{new1'-1}
\|(\varphi_T^n, \psi_T^n)\|_{\mathcal{X}}^2  \leq  C\left( \|(\varphi^n,\psi^n)\|^2_{L^2(0,T;\mathcal{X})}   +  \|  \psi^n(\cdot, L) \|_{L^2(0,T)}^2  +  \| \varphi_x^n(\cdot, L) + a\psi_x^n(\cdot, L) \|_{L^2(0,T)}^2 \right),
\end{equation}
where $C=\max \{ K, \frac{1}{T}, \frac{1}{c}, \frac{b}{c} \}$, with $K^{-1}=\min \left\lbrace 1, \left( 1-   \left(\frac{1-a^2b}{c}\right)C_1\right) \right\rbrace >0$. Proceeding as in the proof of Proposition \ref{prop1} we obtain the convergence of terms on the right-hand side of
\eqref{new1'-1} and conclude that $\{(\varphi_T^n,\psi_T^n)\}_{n\in \N}$ is a Cauchy sequence in $\mathcal{X}$. If we denote its limit by $(\varphi_T,\psi_T)$,
from \eqref{normalized_1'} we get
\begin{equation}\label{norma 1}
\|(\varphi_T,\psi_T)\|_{\mathcal{X}}=1.
\end{equation}
And due to \eqref{new1'}, Proposition \ref{welladj} and compactness argument, we can prove that the corresponding solution \((\varphi, \psi)\)  of
\begin{equation}\label{new5-1}
\begin{cases}
\varphi_t + \varphi_{xxx} + a\psi_{xxx}  =0,  & \text{in $(0,L) \times (0,T)$,} \\
c \psi_t + ab\varphi_{xxx} + r \psi_{x} + \psi_{xxx} =0, & \text{in $(0,L) \times (0,T)$,} \\
\varphi(t,0) =\varphi_x(t,0) =\psi(t,0) =\psi_{x}(t,0) =0, & \text{in $(0,T)$,} \\
 \varphi_{xx}(t,L)+a \psi_{xx}(t,L)=0, & \text{in $ (0,T)$,} \\
 a b\varphi_{xx}(t,L)+{r} \psi(t,L)+ \psi_{xx}(t,L)=0, & \text{in $(0,T)$,} \\
 \varphi(T,x)= \varphi_T(x) \quad \text{and} \quad \psi(T,x)  = \psi_T(x), &  \text{in $(0,L)$,}
\end{cases}
\end{equation}
satisfies the following additional boundary conditions
\begin{equation}\label{new5}
\begin{cases}
 \varphi_{xx}(t,0) +a\psi_{xx}(t,0) = 0 & \text{in $(0,T)$ } \\
\varphi_x(t,L) +a \psi_x(t,L)= 0 & \text{in $(0,T)$,} \\
ab\varphi_{xx}(t,0) + \psi_{xx}(t,0)   =0 & \text{in $(0,T)$.}
\end{cases}
\end{equation}
Notice that due to \eqref{norma 1}, the solutions of \eqref{new5-1}-\eqref{new5} cannot be identically zero. However, arguing as in the proofs of Lemmas \ref{lemma1} and \ref{ucp1}, we can conclude
that $(\varphi,\psi)=(0,0)$, which drives us to a contradiction. Indeed, in this case, the problem is reduced to proving the solution of the stationary problem
\begin{equation}\label{auxsys_1'}
\begin{cases}
\lambda \varphi + \varphi_{xxx} + a\psi_{xxx}=0,  \\
\lambda c \psi +r\psi_x+ab\varphi_{xxx} +\psi_{xxx} =0,
\end{cases}
\end{equation}
satisfying  the boundary conditions
\begin{equation}\label{kdvadjconditionr_1'}
\begin{cases}
\varphi(0) =\varphi_x(0)=0\\
\psi(0) =\psi_{x}(0)  =0,\\
 \varphi_{xx}(L)+a \psi_{xx}(L)=0, \\
 a b\varphi_{xx}(L)+r \psi(L)+ \psi_{xx}(L)=0, \\
  \varphi_{xx}(0) +a\psi_{xx}(0) = 0, \\
\varphi_x(L) +a \psi_x(L)= 0,  \\
ab\varphi_{xx}(0) + \psi_{xx}(0)   =0,
\end{cases}
\end{equation}
for some $\lambda \in \C$, has to be the trivial one. Arguing by contradiction, we  suppose that there exists  $\varphi, \psi$ in $H^{3}(0,L)\setminus \{0\}$, a solution of  \eqref{auxsys_1'}-\eqref{kdvadjconditionr_1'}, and let us denote \(\hat{\varphi}(\xi)=\int_{0}^{L} \mathrm{e}^{-i x \xi} \varphi(x) \mathrm{d} x\) and \(\hat{\psi}(\xi)=\int_{0}^{L} \mathrm{e}^{-i x \xi} \psi(x) \mathrm{d} x\). Then, multiplying the equations by e \(^{-i x \xi},\) integrating by parts over \((0, L)\), using the boundary conditions, taking $\lambda= i p$ with $p\in \C$ and by using the Paley-Wiener Theorem, it follows that $\hat{\varphi}$ and $\hat{\psi}$ are entire functions given  by
\begin{equation}
\hat{\varphi}(\xi)=  e^{-iL\xi} \frac{ B(\xi) }{P(\xi) }		\quad  \text{and} \quad  \nu(\xi)= e^{-iL\xi} \frac{ C(\xi) }{P(\xi) }	,\nonumber
\end{equation}
where 
\begin{gather*}    
P(\xi) =\left(1-a^2 b\right)\xi ^6 - r\xi ^4 -(c+1)  p \xi ^63+pr \xi  +c p^2,\\
B(\xi) =  ia\alpha_2  \xi^5    -a\alpha_1 \xi^4+i\alpha_3\xi^2(pc+r\xi -\xi^3), \\
C(\xi) =  iab\alpha_3\xi^5     -(\alpha_3 \xi + i \alpha_1\xi^2)(p -\xi^3)
\end{gather*}
 and the coefficients $ \alpha_1$, $ \alpha_2$ and $ \alpha_3$ are given by
\begin{equation}
  \alpha_1=   ab \varphi(L)+\psi(L), \quad \alpha_2=  \varphi(L)+a\psi(L),\quad  \alpha_3=  ab\varphi_x(L)+a\psi_x(L).\nonumber
\end{equation}
Note that $\frac{ B(\xi) }{P(\xi) }$ and $\frac{ C(\xi) }{P(\xi) }$ can not be entire functions.
Indeed,  $B(\xi)$ and $C(\xi)$ are polynomials of degree five, and $P(\xi)$ is a polynomial of degree six. Then, it drives us to a contradiction and the proof ends.
\end{proof}
\subsubsection{\textbf{Further Configuration of Three Controls}}
Following the steps of the proof of Proposition \ref{prop2}, we also obtain similar results for the configuration of the following controls:
\begin{equation}\label{another6}
\left\lbrace
\begin{array}{llll}
u(t,0)=h_0(t), & u_x(t,L)=0, & u_{xx}(t,L)=0 & \text{in $(0,T)$},\\
v(t,0)=g_0(t), & v_x(t,L)=g_1(t), & v_{xx}(t,L)=0, & \text{in $(0,T)$}.
\end{array}\right.
\end{equation}
For the case of \textbf{controls configuration  \eqref{another6}}, the observability inequality is given by
\begin{multline}\label{eq2}
\| (\varphi_T,\psi_T)\|^2_{\mathcal{X}}  \le C \left( \|  \varphi_{xx}(t,0) +a\psi_{xx}(t,0) \|^{2}_{H^{-\frac13}(0,T)}
\right. \\
\left. + \|   ab\varphi_x(t,L) +  \psi_x(t,L) \|^{2}_{L^2(0,T)}   + \|  ab\varphi_{xx}(t,0) + \psi_{xx}(t,0)  \|^{2}_{H^{-\frac13}(0,T)}\ \right)
\end{multline}
provided that relation \eqref{putaquepario3} holds. The proof of \eqref{eq2} follows closely the proof of \eqref{obser3}.
\section{Exact Controllability}\label{exactcontrollability}
In this section, we first use the observability inequalities proved in the previous section to study the linear controllability of the system  \eqref{kdvlin}-\eqref{kdvlinconditions}  with different combinations of boundary controls. Next, the local controllability of the full system
is derived by employing a fixed point argument.

We initiate our study assuming that the parameter $r\neq 0$. This case involves a higher difficulty since we need to consider a new
strategy for dealing with the whole system. Indeed, as pointed out in the proof of Proposition \ref{linearNonr}, when $r=0$ it is possible to investigate the controllability-observability of the linear system by studying a single KdV equation. We discuss the case where $r= 0$ in section \ref{furthercommentsopenproblems}.

\subsection{Linear System}
In this subsection, we study the controllability of the linearized system,
\begin{equation}\label{linear1}
\begin{cases}
u_t + uu_x+u_{xxx} + av_{xxx}  =0, \\
c v_t +r v_x +vv_x+ab  u_{xxx} +v_{xxx} =0,\\
u(0,x)=u_{0}(x), \quad v(0,x)=v_{0}(x),
\end{cases}
\end{equation}
with the boundary conditions
\begin{equation}\label{linear2}
\left\lbrace\begin{gathered}
u(t,0)=h_0(t),\,\,u_x(t,L)=h_1(t),\,\,u_{xx}(t,L)=h_2(t),\\
v(t,0)=g_0(t),\,\,v_x(t,L)=g_1(t),\,\,v_{xx}(t,L)=g_2(t).\\
\end{gathered}\right.
\end{equation}
Multiplying the system  \eqref{kdvlin}-\eqref{kdvlinconditions} by a solution of the adjoint system  \eqref{kdvadjoin}-\eqref{kdvadjcondition}, we find an  equivalent condition for the exact controllability property:
\begin{lemma}\label{exactlema}
For any \(\left(u^{1}, v^{1}\right)\) in \(\mathcal{X},\) there exist four controls \(\vec{h}=\left(h_{0}, h_{1}, h_2\right)\) and \(\vec{g}=\left(g_{0}, g_{1}, g_2\right)\) in \(\mathcal{H}\) such that the solution \((u, v)\) of  \eqref{kdvlin}-\eqref{kdvlinconditions}  satisfies $(u(T,x),v(T,x))=(u_1(x),v^1(x))$ if and only if
\begin{multline}\label{eqrnoz}
\frac{b}{c}\int_0^L u(T,x)\varphi(T,x)dx + \int_0^L  v(T,x)\psi(T,x)dx = - \frac{b}{c}\left\langle h_2(\cdot), \left( \varphi(\cdot,L) +a\psi(\cdot,L) \right) \right\rangle_{H^{-\frac13},H^{\frac13}}  + \\
\frac{b}{c}\int_0^T h_1(t) \left( \varphi_x(t,L) +a \psi_x(t,L) \right) dt +\frac{b}{c}\left\langle h_0(\cdot), \left(  \varphi_{xx}(\cdot,0) +a \psi_{xx}(\cdot,0) \right) \right\rangle_{H^{\frac13},H^{-\frac13}}+\\
 - \frac{1}{c}\left\langle g_2(\cdot), \left( ab \varphi(\cdot,L) + \psi(\cdot,L) \right)\right\rangle_{H^{-\frac13}, H^{\frac13}}   +
\frac{1}{c}\int_0^T g_1(t) \left(ab \varphi_x(t,L) + \psi_x(t,L) \right) dt+\\
\frac{1}{c}\left\langle g_0(\cdot), \left(ab\varphi_{xx}(\cdot,0) +
 \psi_{xx}(\cdot,0) \right) \right\rangle_{H^{\frac13},H^{-\frac13}},
\end{multline}
for any \(\left(\varphi_T, \psi_T\right)\) in \(\mathcal{X},\) where \((\varphi, \psi)\) is the solution of the backward system \eqref{kdvadjoin}-\eqref{kdvadjcondition} with initial data \(\left(\varphi_T, \psi_T\right)\).
\end{lemma}
The next result gives a positive answer to the linear control problem associated to \eqref{def1}.
\begin{theorem}\label{mainlineartheorem1}
 Let \(T>0\) and \(L>0 .\) Then,  system \eqref{linear1}-\eqref{linear2} is exactly controllable in time \(T\) under the following four controls configurations
\begin{enumerate}
\item[(i)]  $ h_0\neq 0, h_1\neq 0, h_2\neq 0 \quad \text{and} \quad  g_0= 0, g_1\neq 0, g_2=0$,
\item[(ii)]  $ h_0= 0, h_1\neq 0, h_2= 0 \quad \text{and} \quad  g_0\neq 0, g_1\neq 0, g_2\neq0$,
\item[(iii)]  $ h_0\neq 0, h_1\neq 0, h_2= 0 \quad \text{and} \quad  g_0\neq 0, g_1\neq 0, g_2=0$,
\item[(iv)]  $ h_0= 0, h_1\neq 0, h_2\neq 0 \quad \text{and} \quad  g_0= 0, g_1\neq 0, g_2\neq0$.
\end{enumerate}
Moreover, if the parameters $a,b,c$ satisfy
\begin{equation*}
0 <	C_1(1-a^2b)<c,
\end{equation*}
 where $C_1=C_1(T,L)>0$ is the hidden regularity constant given in \eqref{putaquepario2}, the same result is achieved for the following configuration for three controls:
\begin{enumerate}
\item[(i)]  $ h_0\neq 0, h_1\neq 0, h_2= 0 \quad \text{and} \quad  g_0\neq 0, g_1= 0, g_2=0$,
\item[(ii)]  $ h_0\neq 0, h_1= 0, h_2= 0 \quad \text{and} \quad  g_0\neq 0, g_1\neq 0, g_2=0$.
\end{enumerate}
\end{theorem}
\begin{proof}
We will prove the theorem for the case of four controls with the configuration $ h_0\neq 0, h_1\neq 0, h_2\neq 0 \quad \text{and} \quad  g_0= 0, g_1\neq 0, g_2=0$. The other cases are similar, therefore, we omit them. Let us denote by $\Gamma$ the bounded linear  map  defined by
\begin{equation*}
\begin{tabular}{r c c c}
$\Gamma :$ & $L^2(0, L) \times L^2(0, L)$        &  $\longrightarrow$ & $L^2(0, L) \times L^2(0, L)$ \\
           & $(\varphi_T(\cdot), \psi_T(\cdot))$ &  $\longmapsto$     & $\Gamma(\varphi^1(\cdot), \psi^1(\cdot))=(u(T,\cdot), v(T,\cdot))$,
\end{tabular}
\end{equation*}
where $(u,v)$ is the solution of \eqref{linear1}-\eqref{linear2}  with
\begin{equation}\label{controldefinition}
\begin{cases}
g_0(t) = c\left(ab\varphi_{xx}(t,0)+\psi_{xx}(t,0)\right), & h_0(t)=0, \\
 g_1(t)= c\left( a\varphi_x(t,L)+\frac{1}{c}\psi_x(L,t)\right),   & h_1(t)=\frac{c}{b}\left(\varphi_x(t,L)+a\psi_x(t,L)\right), \\
g_2(t)=-c\left(ab\varphi(t,L)+\psi(t,L)\right), & h_2(t)=0, \\
\end{cases}
\end{equation}
and $(\varphi,\psi)$ the solution of the backward system \eqref{kdvadjoin}-\eqref{kdvadjcondition} with initial data \(\left(\varphi_T, \psi_T\right)\).  According to Lemma \ref{exactlema} and Proposition \ref{prop1}, we obtain
\begin{align*}
\left (\Gamma(\varphi_T, \psi_T), (\varphi_T, \psi_T) \right)_{\mathcal{X}} =&\left\|\varphi_x(\cdot,L)+a\psi_x(\cdot,L)\right\|_{L^2(0,T)}^2 +\left\|ab\varphi_x(\cdot,L)+\psi_x(\cdot,L)\right\|_{L^2(0,T)}^2  \\
&+\left\|\left(ab\varphi_{xx}(\cdot,0)+\psi_{xx}(\cdot,0)\right)\right\|_{H^{-\frac13}(0,T)}^2 \\
&+ \left\|\left(ab\varphi(\cdot,L)+\psi(\cdot,L)\right)\right\|_{H^{\frac13}(0,T)}^2 \\
\geq & \, C^{-1} \|(\varphi_T, \psi_T)\|_{\X}^2.
\end{align*}
Thus, by the Lax-Milgram theorem, \(\Gamma\) is invertible. Consequently, for given \(\left(u_T, v_T\right) \in \mathcal{X},\) we can define \(\left(\varphi_T, \psi_T\right):=\Gamma^{-1}\left(u_T, v_T\right)\) to solve the system \eqref{kdvadjoin}-\eqref{kdvadjcondition} with initial data \(\left(\varphi_T, \psi_T\right)\) and get \((\varphi, \psi) \in X_{T} .\) Then, if $(h_0,h_1,h_2)$ and $(g_0,g_1,g_2)$ are given by \eqref{controldefinition}, the corresponding solution \((u, v)\) of the system \eqref{linear1}-\eqref{linear2} satisfies
$$
(u(0,\cdot), v(0,\cdot))=(0,0) \quad \text { and } \quad(u(T,\cdot), v(T,\cdot))=\left(u^{1}(\cdot),v^{1}(\cdot)\right).
$$
The result follows from Lax–Milgram theorem. 
\end{proof}
As a consequence of the previous analysis, we have the following result:
\begin{proposition}\label{controlOperator}
For any $L>0$ (when the linear controllability is guaranteed), there exists a bounded linear operator
$$\Psi:\X \times \X \ \to \  \mathcal{H} \times \mathcal{H} $$
such that,  for any $(u_0,v_0)\in \X$ and  $(u_T,v_T) \in \X$,
$$\Psi\left(  \left(u_0,v_0\right), \left(u_T,v_T\right)  \right):= \left(\vec{h},\vec{g}\right)$$
where $\vec{h}=(h_0,h_1,h_2)$ and $\vec{g}=(g_0,g_1,g_2)$ are the controls configurations given by Theorem \ref{mainlineartheorem1}. Then, the  system   \eqref{kdvlin}-\eqref{kdvlinconditions}
 admits a unique solution $(u,v) \in   X_T= {C\left([0,T];\mathcal{X} \right)} 	\cap L^2(0,T;(H(0,L))^2)$,  such that
$$ u(\cdot,0)=u_0,\  v(\cdot,0)=v_0 \quad \text{and} \quad u(T,\cdot)=u_T, \ v(T,\cdot)=v_T.$$
\end{proposition}
\subsection{Nonlinear System}
In this section, we study the controllability of the full system,
\begin{equation}
\begin{cases}\tag{\ref{kdv}}
u_t + uu_x+u_{xxx} + av_{xxx} + a_1vv_x+a_2 (uv)_x =0, \\
c v_t +r v_x +vv_x+ a b u_{xxx} +v_{xxx}+ a_2 buu_x+ a_1 b(uv)_x  =0,\\
u(0,x)=u_{0}(x), \quad v(0,x)=v_{0}(x),
\end{cases}
\end{equation}
with boundary conditions
\begin{equation*}
\left\lbrace
\begin{gathered}\tag{\ref{inputs}}
u(t,0)=h_0,\,\,u_x(t,L)=h_1(t),\,\,u_{xx}(t,L)=h_2(t),\\
v(t,0)=g_0,\,\,v_x(t,L)=g_1(t),\,\,v_{xx}(t,L)=g_2(t).\\
\end{gathered}\right.
\end{equation*}
According to Proposition \ref{linearNonr} and using the definition of the boundary solution, $S_{bdr}$, presented in \cite{KRZ}, the solution of 	\eqref{kdv}-\eqref{inputs} can be written as
\begin{gather*}
\left( \begin{array}{cc} u\\v \end{array}\right) =  S(t) \left( \begin{array}{cc} u_0\\v_0 \end{array}\right) + S_{bdr}(t) \left(\begin{array}{cc}\vec{h} \\  \vec{g}  \end{array}\right) - \int_0^t S (t-\tau)
\left( \begin{array}{cc}  {a_1} vv_x+{a_2} (uv)_x \\ \frac{r}{c}v_x +   \frac{ a_1 b}{c} uu_x+\frac{a_2 b}{c}(uv)_x +  \end{array}\right) d\tau,
\end{gather*}	
where $\{S(t)\}_{t\geq 0}$ and $\{S_{bdr}(t)\}_{t\geq 0}$ are a $C_0$-semigroup and a boundary differential operator, respectively.

\begin{proof}[\textbf{Proof of Theorem \ref{maintheorem}}]
 We will treat the nonlinear problem \eqref{kdv}-\eqref{inputs}
 by using a classical fixed point argument. Indeed, for $u,v \in X_T$,  let us introduce the notation
$$\left(\begin{array}{cc} \upsilon \\ \nu (T,u,v)\end{array}\right)  := \int_0^T S(T-\tau) \left( \begin{array}{cc}  a_1  vv_x  + a_2 (uv)_x  \\
\frac{ a_1 b}{c}  uu_x + \frac{ a_2 b}{c} (uv)_x  \end{array} \right)  d\tau.  $$
Then, for any $u_0,v_0 \in L^2(0,T)$, Proposition \ref{controlOperator} guarantees that there exists a  mapping $\Psi$  which allows us to choose
\begin{gather*}  \left( \begin{array}{cc} \vec{h} \\ \vec{g} \end{array}\right)  = \Psi\left(    \left( \begin{array}{cc}   u_0\\v_0  \end{array}\right),  \left( \begin{array}{cc}  u_T \\ v_T \end{array}\right) +  \left( \begin{array}{cc} \upsilon \\  \nu  (T,u,v)\end{array}\right)   \right).\end{gather*}
Thus, if we define the operator $\Gamma: X_T  \to X_T $ as follows
\begin{multline*}
\Gamma \left( \begin{array}{cc} u\\v \end{array}\right)  =   S(t) \left( \begin{array}{cc} u_0\\v_0 \end{array}\right) + S_{bdr}(x)  \Psi\left(    \left( \begin{array}{cc}   u_0\\v_0  \end{array}\right),  \left( \begin{array}{cc}  u_T \\ v_T \end{array}\right) +  \left( \begin{array}{cc} \upsilon \\  \nu \end{array}\right)  (T,u,v) \right) \\ - \int_0^t S (t-\tau)
\left( \begin{array}{cc}  a_1  vv_x  + a_2 (uv)_x  \\
\frac{r}{c}v_x +\frac{ a_1 b}{c}  uu_x + \frac{a_2 b}{c} (uv)_x  \end{array} \right)  d\tau,
\end{multline*}
it is clear that
\begin{equation*}
\begin{cases}
\left. \Gamma \left( \begin{array}{cc} u\\v \end{array}\right) \right|_{t=0}=\left( \begin{array}{cc} u_0\\v_0 \end{array}\right), \\
\\
 \left. \Gamma \left( \begin{array}{cc} u\\v \end{array}\right) \right|_{t=T}=	\left[	\left( \begin{array}{cc}  u_T \\ v_T \end{array}\right) +  \left( \begin{array}{cc} \upsilon \\  \nu  (T,u,v)\end{array}\right)    \right] -  \left( \begin{array}{cc} \upsilon \\  \nu  (T,u,v)\end{array}\right) =   \left( \begin{array}{cc}  u_T \\ v_T \end{array}\right).
 \end{cases}
\end{equation*}
We claim that the operator $\Gamma$ has a unique fixed point. To do that, we will prove that $\Gamma$ is a contraction in the space $X_T$. Indeed, let us consider
\begin{equation}
B_r=\left\{ \left( \begin{array}{cc} u\\v \end{array}\right)  \in  X_T : \left\| \left( \begin{array}{cc} u\\v \end{array}\right)\right\|_{X_T} \le r \right\},
\end{equation}
for some $r>0$ to be chosen later.   Proposition \ref{controlOperator} and	Theorem	\ref{nonlinearwell-posedness} imply that there exist  constants $C_1,C_{2}>0$, such that
\begin{multline*}
\left\|\Gamma\left( \begin{array}{cc} u\\v \end{array}\right) \right\|_{\mathcal{X}}  \le   C_1 \left\lbrace  \left\| S(t) \left( \begin{array}{cc} u_0\\v_0 \end{array}\right) 	\right\|_{\X}
 +\left\| S_{bdr}(x)  \Psi\left(    \left( \begin{array}{cc}   u_0\\v_0  \end{array}\right),  \left( \begin{array}{cc}  u_T \\ v_T \end{array}\right) +  \left( \begin{array}{cc} \upsilon \\  \nu (T,u,v) \end{array}\right)  \right) \right \|_{\X} \right.
 \\
 \left.+\left\| \int_0^t S (t-\tau) \left( \begin{array}{cc}  a_1  vv_x  + a_2 (uv)_x  \\
\frac{r}{c}v_x+\frac{ a_1 b}{c}  uu_x + \frac{ a_2 b}{c} (uv)_x  \end{array} \right)  d\tau \right\|_{\X} \right\rbrace
\end{multline*}
and
\begin{multline*}
\left\|   \Psi\left(    \left( \begin{array}{cc}   u_0\\v_0  \end{array}\right),  \left( \begin{array}{cc}  u_T \\ v_T \end{array}\right) +  \left( \begin{array}{cc} \upsilon \\  \nu \end{array}\right)  (T,u,v) \right) \right \|_{\mathcal{H}} \\
\leq
C_2 \left\lbrace    \left\|  \left( \begin{array}{cc}   u_0\\v_0  \end{array}\right) \right\|_{\X}+ \left\| \left( \begin{array}{cc}  u_T \\ v_T \end{array}\right) \right\|_{\X} +  \left\|\left( \begin{array}{cc} \upsilon \\  \nu (T,u,v) \end{array}\right)  \right\|_{\X} \right\rbrace.
\end{multline*}
From  Lemma \cite[Lemma 3.1]{BSZ}, we  deduce that
$$
\left\|\Gamma\left(\begin{array}{l}
{u} \\
{v}
\end{array}\right)\right\|_{X_{T}} \leq C_{3} \delta+C_{4}(r+1) r,
$$
where \(C_{4}\) is a constant depending only on \(T\). Thus, choosing \(r\) and \(\delta\) satisfying
$$
r=2 C_{3} \delta
$$
and
$$
2 C_{3} C_{4} \delta+C_{4} \leq \frac{1}{2}
$$
the operator \(\Gamma\) maps \(B_{r}\) into itself for any \(\left( \begin{array}{cc} u\\v \end{array}\right) \in \mathcal{X}_{T}\). Now, proceeding as in the proof of Theorem \ref{nonlinearwell-posedness}, we  obtain that
$$
\left\|\Gamma\left(\begin{array}{l}
{u} \\
{v}
\end{array}\right)-\Gamma\left(\begin{array}{l}
{\widetilde{u}} \\
{\tilde{v}}
\end{array}\right)\right\|_{X_{T}} \leq C_{5}(r+1) r\left\|\left(\begin{array}{c}
{u-\widetilde{u}} \\
{v-\widetilde{v}}
\end{array}\right)\right\|_{X_{T}},
$$
for any \(\left( \begin{array}{cc} u\\v \end{array}\right), \left( \begin{array}{cc} \widetilde{u}\\ \widetilde{v} \end{array}\right) \in B_{r}\), where \(C_{5}>0\) depends only on \(T .\) Thus, taking \(\delta>0,\) such that
$$
\gamma=2 C_{3} C_{5} \delta+C_{5}<1
$$
we have
$$
\begin{array}{l}
{\qquad\left\|\Gamma\left(\begin{array}{c}
{u} \\
{v}
\end{array}\right)-\Gamma\left(\begin{array}{c}
{\widetilde{u}} \\
{v}
\end{array}\right)\right\|_{X_{T}} \leq \gamma\left\|\left(\begin{array}{c}
{u-\widetilde{u}} \\
{v-\widetilde{v}}
\end{array}\right)\right\|_{X_{T}}}.
\end{array}
$$
Therefore, the map \(\Gamma\) is a contraction. Thus,  Banach fixed-point theorem guarantees that \(\Gamma\) has a fixed point in \(B_{r}\), and its fixed point is the desired solution.
\end{proof}

\section{ Open Problem}\label{furthercommentsopenproblems}
This section presents some open problems that bring new challenges to the mathematical community. 

\begin{open}
When  $\mathbf{r\neq 0}$ in the control system \eqref{kdv}-\eqref{inputs}, the ideas contained in this work suggest that one can remove controls in several situations. For instance, if we consider the following configuration
\begin{equation}\label{controltresopen}
\left\lbrace
\begin{array}{l l l}
u(t,0)=h_0(t), & u_x(t,L)=h_1(t),& u_{xx}(t,L)=h_2(t),\\
v(t,0)=0,& v_x(t,L)=0, &v_{xx}(t,L)=0,\\
\end{array}
\right.
\end{equation}
 The observability associated with this case is given by 
\begin{multline*}
\| \psi_T\|^2_{L^2(0,L)} + \frac{b}{c} \| \varphi_T\|^2_{L^2(0,L)} \le C \left( \|  \varphi(t,L) +a \psi(t,L) \|^{2}_{H^{\frac13}(0,T)}
\right. \\
\left. + \|  \varphi_x(t,L) + a\psi_x(t,L) \|^{2}_{L^2(0,T)}   + \|  \varphi_{xx}(t,0) + a\psi_{xx}(t,0)  \|^{2}_{H^{-\frac13}(0,T)}\ \right).
\end{multline*}
Arguing by contradiction and proceeding as in the proof of Proposition \ref{prop2}, the problem is reduced  
to proving a unique continuation property for the solutions of a spectral system. In this case, the
Fourier and Paley-Wiener approach leads to the study of the following entire functions
\begin{equation*}
\hat{\varphi}(\xi)= \frac{    ia \xi^3 \left(   e^{-iL\xi}  \left( i	\xi \alpha +\beta \xi^2\right)   - \gamma \right)  }{P(\xi) }		\quad  \text{and} \quad  \hat{\psi}(\xi)= \frac{    i(p- \xi^3) \left(   e^{-iL\xi}  \left( i	\xi \alpha +\beta \xi^2\right)   - \gamma  	\right) }{ P(\xi) }
\end{equation*}
and the nature of the roots of $ e^{-iL\xi}\left( i\alpha \xi	 + 	\beta \xi^2 	 \right)   - \gamma $ is more complicated than the roots of \eqref{entirefour1}. Therefore, the following question arises:\\

{\bf Question 1}\\
Is the full system \eqref{kdv}-\eqref{inputs} locally exactly controllable when three or fewer control inputs act on boundary conditions,  without any restriction on the parameters $a$, $b$, and $c$?
\end{open}

\begin{open}
On the other hand, the controllability results have been established when $r	\neq 0$ are local. It means one can only guide a small amplitude initial state to a small amplitude terminal state by choosing appropriate boundary control inputs. However, even the global well-posedness of the IBVP \eqref{kdv}-\eqref{inputs} is still an open question, even for a homogeneous boundary. Thus, the following question remains open:\\
{\bf Question 1}\\
Is it possible to show the global well-posedness of the IBVP \eqref{kdv}-\eqref{inputs} in $H^s$ for some $s$?

\end{open}

\textbf{Acknowledgements:}
The first author was partially supported by Universidad Nacional de Colombia Sede Manizales, under the projects Nro HERMES-61029 and HERMES-63945. AFP was partially supported by CNPq (Brazil). This work was
proposed when the second author was visiting IMPA (Brazil) with a post-doc fellowship from CNPq (Brazil) and finished with the project CI - 71011  of the Universidad del Valle, Cali-Colombia. This work was carried out during some visits of the first author to the Federal University of Rio de Janeiro. The author would like to thank the University for the hospitality.


\appendix

\section{Proof of Some Results}\label{moreresults}

\begin{proof}[Proof of Proposition \ref{linearNonr-1}]
Let $K_T:= \left\lbrace (u,v) \in X_T: (u,v) \in  L^{\infty}_x(0,L;(H^{\frac{1-k}{3}}(0,T))^2), k=0,1,2\right\rbrace$ equipped with the norm
\[
\|(u,v)\|_{K_T} = \|(u,v)\|_{X_T}+\sum_{k=0}^{2}\|(\partial_x^j  u, \partial_x^j v)\|_{L^{\infty}_x(0,L;(H^{\frac{1-k}{3}}(0,T))^2)}.
\]
Let $0< \beta \leq T$ to be determined later. For each $u,v \in K_{\beta}$, consider the problem
\begin{equation}\label{e1}
\left\lbrace \begin{tabular}{l l}
$\omega_t + \omega_{xxx} + a\eta_{xxx}  =0$, & in $(0,L)\times (0,\beta)$,\\
$\eta_t +\frac{ab}{c}\omega_{xxx} + \frac{1}{c}\eta_{xxx} =-\frac{r}{c}v_x$, & in $(0,L)\times (0,\beta)$,\\
$\omega(0,t) = h_0(t),\,\,\omega_x(t,L) = h_1(t),\,\,\omega_{xx}(t,L) = h_2(t)$,& in $(0,\beta)$,\\
$\eta(0,t) =g_0(t),\,\,\eta_x(t,L) = g_1(t),\,\, \eta_{xx}(t,L) = g_2(t)$,& in $(0,\beta)$,\\
$\omega(0,x)=u^0(x), \quad v(0,x) = v^0(x)$, & in $(0,L)$.
\end{tabular}\right.
\end{equation}
According to Proposition \ref{linearNonr}, we can define the operator
\[
\Gamma: K_{\beta} \rightarrow K_{\beta}, \quad \text{given by} \quad \Gamma(u,v)=(\omega,\eta),
\]
where $(\omega,\eta)$ is the solution of \eqref{e1}. Moreover,
\begin{equation}\label{hr3}
\|\Gamma(u,v)\|_{K_\beta}\leq C\left\lbrace \|(u^0,v^0)\|_{\X}+\|(\overrightarrow{h},\overrightarrow{g})\|_{\mathbf{\mathcal{H}}^2}+\|(0,v_x)\|_{L^1(0,\beta;\X)} \right\rbrace,
\end{equation}
where the positive  constant $C$ depends only on $T$. Since
$$ \|(0,v_x)\|_{L^1(0,\beta;\X)}\leq \beta^{\frac{1}{2}}\|(u,v)\|_{K_\beta},$$
we obtain a positive constant $C>0$, such that
\begin{equation}\label{e2}
\|\Gamma(u,v)\|_{K_\beta}\leq C\left\lbrace \|(u^0,v^0)\|_{\X}+\|(\overrightarrow{h},\overrightarrow{g})\|_{\mathbf{\mathcal{H}}^2}\right\rbrace+C\beta^{\frac{1}{2}}\|(u,v)\|_{K_\beta}.
\end{equation}
Let $(u,v) \in B_{r}(t,0):=\left\lbrace (u,v) \in K_{\beta}: \|(u,v)\|_{K_{\beta}}\leq r\right\rbrace$, with $$r=2C\left\lbrace \|(u^0,v^0)\|_{\X}+\|(\overrightarrow{h},\overrightarrow{g})\|_{\mathbf{\mathcal{H}}^2}\right\rbrace.$$
Choosing $\beta>0$, satisfying
\begin{equation}\label{beta}
C\beta^{\frac{1}{2}}\leq \frac{1}{2},
\end{equation}
from (\ref{e2}) we obtain
$$\|\Gamma(u,v)\|_{K_\beta}\leq r.$$
The above estimate allows us to conclude that
\[
\Gamma: B_r(t,0)\subset K_{\beta} \rightarrow B_r(t,0).
\]
On the other hand, note that $\Gamma(u_1,v_1)-\Gamma(u_2,v_2)$ solves the following system
\begin{equation*}
\left\lbrace \begin{tabular}{l l}
$\omega_t + \omega_{xxx} + a\eta_{xxx}  =0$, & in $(0,L)\times (0,\beta)$,\\
$\eta_t +\frac{ab}{c}\omega_{xxx} + \frac{1}{c}\eta_{xxx} =-\frac{r}{c}(v_{1x}-v_{2x}) $, & in $(0,L)\times (0,\beta)$,\\
$\omega(t,0) = \omega_x(t,L) = \omega_{xx}(t,L) = 0$,& in $(0,\beta)$,\\
$\eta(t,0)=\eta_x(t,L) = \eta_{xx}(t,L) = 0$,& in $(0,\beta)$,\\
$\omega(0,x)=0, \quad v(0,x) = 0$, & in $(0,L)$.
\end{tabular}\right.
\end{equation*}
Again, from Proposition \ref{linearNonr} and \eqref{beta}, we have
\begin{equation*}
\begin{array}{l}
\|\Gamma(u_1,v_1)-\Gamma(u_2,v_2)\|_{K_\beta}\leq C\|(0,v_{1x}-v_{2x})\|_{L^1(0,\beta;\X)}\leq C\beta^{\frac{1}{2}}\|(u_1,v_1)-(u_2,v_2)\|_{K_{\beta}}\\
\qquad\qquad\qquad\qquad\qquad\quad\leq \displaystyle\frac{1}{2}\|(u_1,v_1)-(u_2,v_2)\|_{K_{\beta}}.\nonumber
\end{array}
\end{equation*}
Hence, $\Gamma: B_r(t,0) \rightarrow B_r(t,0)$ is a contraction and, by Banach fixed point theorem, we obtain a unique $(u,v) \in B_r(t,0)$, such that $$\Gamma(u,v) = (u,v) \in K_{\beta},$$
 and \eqref{hr3} holds, for all $t \in (0,\beta)$. Since the choice of $\beta$ is independent of $(u^0,v^0)$, the standard continuation extension argument yields that the solution $(u,v)$ belongs to $K_T$. The proof is complete.
\end{proof}

\begin{proof}[Proof of Theorem \ref{nonlinearwell-posedness}]
Let us consider
$$
\mathcal{F}_{T}=\left\{(u, v) \in X_{T}:(u, v) \in L_{x}^{\infty}\left(0, L ;\left(H^{\frac{1-k}{3}}(0, T)\right)^{2}\right), k=0,1,2\right\},
$$
which is a Banach space equipped with the norm
$$
\|(u, v)\|_{\mathcal{F}_{T}}=\|(u, v)\|_{\mathcal{Z}_{T}}+\sum_{k=0}^{2}\left\|\left(\partial_{x}^{k} u, \partial_{x}^{k} v\right)\right\|_{L_{x}^{\infty}(0, L ;(H^{\frac{1-k}{3}}(0, T))^{2})}.
$$
Let \(0<T^{*} \leq T\) to be determined later. For each \(u, v \in \mathcal{F}_{T^{*}},\) consider the problem
\begin{equation}\label{eq1}
\left\{\begin{array}{ll}
{\omega_{t}+\omega_{x x x}+a \eta_{x x x}=f(u, v),} & {\text { in }(0, L) \times\left(0, T^{*}\right)} \\
{\eta_{t}+\frac{a b}{c} \omega_{x x x}+\frac{1}{c} \eta_{x x x}=s(u, v),} & {\text { in }(0, L) \times\left(0, T^{*}\right)} \\
{\omega(0, t)=h_{0}(t), \omega_x(L, t)=h_{1}(t), \omega_{xx}(L, t)=h_{2}(t),} & {\text { in }\left(0, T^{*}\right)} \\
{\eta(0, t)=g_{0}(t), \eta_x(L, t)=g_{1}(t), \eta_{xx}(L, t)=g_{2}(t),} & {\text { in }\left(0, T^{*}\right)} \\
{\omega(x, 0)=u^{0}(x), \quad v(x, 0)=v^{0}(x),} & {\text { in }(0, L)},
\end{array}\right.
\end{equation}
where the nonlinear terms are given by $$f(u, v)=-a_{1}\left(v v_{x}\right)-a_{2}(u v)_{x}$$ and $$s(u, v)=-\frac{r}{c} v_{x}-\frac{a_{2} b}{c}\left(u u_{x}\right)-\frac{a_{1} b}{c}(u v)_{x}.$$
Proceeding as in \cite[Lemma 3.1]{BSZ},  we deduce that \(f(u, v)\) and \(s(u, v)\) belong to \(L^{1}\left(0, T^{*} ; L^{2}(0, L)\right)\).  Then, according to Proposition \ref{linearNonr}, we can define the operator
$$
\Gamma: \mathcal{F}_{T^{*}} \rightarrow \mathcal{F}_{T^{*}}, \quad \text { given by } \quad \Gamma(u, v)=(\omega, \eta),
$$
where \((\omega, \eta)\) is the solution of \eqref{eq1}. Moreover, from \cite[Lemma 3.1]{BSZ}, we obtain a positive constant \(C\) depending only on \(T^{*} ,\) such that
\begin{multline*}
\|\Gamma(u, v)\|_{\mathcal{F}_{T^{*}}} \leq C\left\{\left\|\left(u^{0}, v^{0}\right)\right\|_{\mathcal{X}}+\|(\vec{h}, \vec{g})\|_{\mathcal{H}_{T^{*}}}\right\} \\
+C C_{1}\left(\left(T^{*}\right)^{\frac{1}{2}}+\left(T^{*}\right)^{\frac{1}{3}}\right)\left(\|u\|_{X_{{T}^{*}}}^{2}+\left(\|u\|_{X_{{T}^{*}}}+1\right)\|v\|_{X_{T^{*}}}+\|v\|_{X_{T^{*}}}^{2} \right).
\end{multline*}
Let $(u, v) \in B_{r}(t,0):=\left\{(u, v) \in \mathcal{F}_{T^{*}}:\|(u, v)\|_{X_{T^{*}}} \leq r\right\}$, where $$r=2 C\left\{\left\|\left(u^{0}, v^{0}\right)\right\|_{\mathcal{X}}+\|(\vec{h}, \vec{g})\|_{\mathcal{H}_{T}}\right\}.$$
From the estimate above, we get
$$
\|\Gamma(u, v)\|_{\mathcal{F}_{T^{*}}} \leq \frac{r}{2}+C C_{1}\left(\left(T^{*}\right)^{\frac{1}{2}}+\left(T^{*}\right)^{\frac{1}{3}}\right)(3 r+1) r.
$$
Then, we can choose \(T^{*}>0,\) such that
\begin{equation}\label{eq6}
C C_{1}\left(\left(T^{*}\right)^{\frac{1}{2}}+\left(T^{*}\right)^{\frac{1}{3}}\right)(3 r+1) \leq \frac{1}{2},
\end{equation}
to obtain $$\|\Gamma(u, v)\|_{\mathcal{F}_{T^{*}}} \leq r.$$ Thus, we conclude that $
\Gamma: B_{r}(t,0) \subset \mathcal{F}_{T^{*}} \rightarrow B_{r}(t,0)$. On the other hand, \(\Gamma\left(u_{1}, v_{1}\right)-\Gamma\left(u_{2}, v_{2}\right)\) solves the system
$$
\left\{\begin{array}{ll}
{\omega_{t}+\omega_{x x x}+a \eta_{x x x}=f\left(u_{1}, v_{1}\right)-f\left(u_{2}, v_{2}\right),} & {\text { in }(0, L) \times\left(0, T^{*}\right)} \\
{\eta_{t}+\frac{a b}{c} w_{x x x}+\frac{1}{c} \eta_{x x x}=s\left(u_{1}, v_{1}\right)-s\left(u_{2}, v_{2}\right),} & {\text { in }(0, L) \times\left(0, T^{*}\right)} \\
{\omega_{x x}(0, t)=\omega_{x}(L, t)=\omega_{x x}(L, t)=0,} & {\text { in }\left(0, T^{*}\right)} \\
{\eta_{x x}(0, t)=\eta_{x}(L, t)=\eta_{x x}(L, t)=0,} & {\text { in }\left(0, T^{*}\right)} \\
{\omega(x, 0)=0, \quad v(x, 0)=0,} & {\text { in }(0, L)}.
\end{array}\right.
$$
Then, Proposition \ref{linearNonr} and \cite[Lemma 3.1]{BSZ} imply that
$$
\begin{aligned}
\left\|\Gamma\left(u_{1}, v_{1}\right)-\Gamma\left(u_{2}, v_{2}\right)\right\|_{\mathcal{F}_{T^*}} \leq C_{3}\left(\left(T^{*}\right)^{\frac{1}{2}}+\left(T^{*}\right)^{\frac{3}{3}}\right)(8 r+1)\left\|\left(u_{1}-u_{2}, v_{1}-v_{2}\right)\right\|_{\mathcal{F}_{T^{*}}},
\end{aligned}
$$
for some positive constant \(C_{3} .\) Thus, choosing  \(T^{*}>0,\) such that \eqref{eq6} holds and
$$
C_{3}\left(\left(T^{*}\right)^{\frac{1}{2}}+\left(T^{*}\right)^{\frac{1}{3}}\right)(8 r+1) \leq \frac{1}{2},
$$
it follows that
$$
\left\|\Gamma\left(u_{1}, v_{1}\right)-\Gamma\left(u_{2}, v_{2}\right)\right\|_{\mathcal{F}_{T^{*}}} \leq \frac{1}{2}\left\|\left(u_{1}-u_{2}, v_{1}-v_{2}\right)\right\|_{\mathcal{F}_{T^{*}}}.
$$
Hence, \(\Gamma: B_{r}(t,0) \rightarrow B_{r}(t,0)\) is a contraction and, by Banach fixed-point theorem, we obtain a unique \((u, v) \in B_{r}(t,0),\) such that \(\Gamma(u, v)=(u, v) \in \mathcal{F}_{T^{*}}\).
\end{proof}

\color{black}

\end{document}